\documentclass{article}

\usepackage{amsmath, amsthm, amsfonts}
\usepackage[utf8]{inputenc}
\usepackage[colorlinks, linkcolor=blue, citecolor=blue, urlcolor=blue]{hyperref}

\usepackage{srcltx}
\usepackage{hyperref}

\usepackage{bbold} 

\def\1{\mathbb{1}}

\def\R{\mathbb{R}}

\def\LL{{\cal L}}

\def\d{\,\mathrm{d}}
\def\dx{\,\mathrm{d}x}
\def\dy{\,\mathrm{d}y}

\def \fa {\quad \text{ for all }}

\newtheorem{thm}{Theorem}[section]

\newtheorem{lem}[thm]{Lemma}

\newtheorem{hyp}[thm]{Hypothesis}
\theoremstyle{definition}
\newtheorem{dfn}[thm]{Definition}
\theoremstyle{remark}
\newtheorem{rem}{Remark}[section]

\title{Fine asymptotics of profiles and relaxation to equilibrium for
  growth-fragmentation equations with variable drift rates}

\author{
  \textsc{Daniel Balagu\'e}\footnote{Departament de Matem\`atiques,
    Universitat Aut\`onoma de Barcelona, 08193 Bellaterra, Spain. Email:
    \texttt{dbalague@mat.uab.es}}
  \and
  \textsc{Jos\'e A. Ca\~{n}izo}\footnote{School of Mathematics, Watson Building,
  University of Birmingham, Edgbaston, Birmingham B15 2TT, UK. Email:
   \texttt{j.a.canizo@bham.ac.uk}}
  \and
  \textsc{Pierre Gabriel}\footnote{Laboratoire de Math\'ematiques de Versailles, CNRS UMR 8100,
  Universit\'e de Versailles Saint-Quentin-en-Yvelines, 45 Avenue de \'Etats-Unis, 78035~Versailles~cedex, France.
  Email: \texttt{pierre.gabriel@uvsq.fr}}
}

\begin{document}

\maketitle

\begin{abstract}
  We are concerned with the long-time behavior of the
  growth-frag\-men\-ta\-tion equation. We prove fine estimates on the
  principal eigenfunctions of the growth-fragmentation operator,
  giving their first-order behavior close to $0$ and $+\infty$. Using
  these estimates we prove a spectral gap result by following the
  technique in~\cite{MR2832638}, which implies that solutions decay to the
  equilibrium exponentially fast. The growth and fragmentation
  coefficients we consider are quite general, essentially only assumed
  to behave asymptotically like power laws.
\end{abstract}

\


\

\noindent{\bf AMS Class. No.} 35B40, 45C05, 45K05, 82D60, 92C37

\

\tableofcontents

\section{Introduction}

In this paper we are interested in the long-time behavior of the
growth-frag\-men\-ta\-tion equation, commonly used as a model for cell
growth and division and other phenomena involving fragmentation
\cite{MR2270822,MR0860959,GabrielPhD}. There are a number of works which study the
existence and other properties of the first eigenfunctions (also
called \emph{profiles}) of the growth-fragmentation operator and its
dual \cite{MR2250122,MR2652618,MR2114413} and the convergence of solutions to
equilibrium \cite{MR2536450,MR2114128,MR2162224,MR2065377,MR2832638,PS}. These
eigenfunctions are fundamental since they give the asymptotic shape of
solutions (i.e., the stationary solution of the rescaled equation) and
a conserved quantity of the time evolution. However, precise estimates
on their behavior close to $0$ and $+\infty$ are usually not given,
are very rough, or are restricted to a particular kind of growth or
fragmentation coefficients. Our first objective is to give accurate
estimates on the first eigenfunctions, valid for a wide range of
growth and fragmentation coefficients which include most cases in
which they behave like power laws. We give, in most cases, the
first-order behavior of both first eigenfunctions (of the
growth-fragmentation operator and its dual); detailed results are
given later in this introduction.

Our second objective is to use these estimates to show that the
growth-fragmenta\-tion operator has a spectral gap (in a certain natural
Hilbert space) for a wide choice of the coefficients, which is
interesting because it readily implies exponential convergence to
equilibrium of solutions. For this we follow the techniques in
\cite{MR2832638}, which require careful estimates on the profiles which
were previously available only for particular growth rates (constant
and linear). Our results on exponential convergence to equilibrium are
valid for general coefficients behaving like power laws, improving or
complementing known results applicable to constant or linear total
fragmentation rates \cite{MR2536450,MR2114128,MR2832638}. However,
our results still impose some restrictions on the fragment
distribution (which must be bounded below) and the decay of the total
fragmentation rate for small sizes.

Let us introduce the equation under study more precisely and state our
main results. The \emph{growth-fragmentation equation} is given by
\begin{subequations}
  \label{eq:growth-frag}
  \begin{gather}
    \label{eq:growth-frag-eq}
    \partial_t g_t(x) + \partial_x (\tau(x) g_t(x)) + \lambda g_t(x)
    = \LL [g_t](x),
    \\
    \label{eq:growth-frag-boundary}
    (\tau g_t)(0) = 0 \qquad (t \geq 0),
    \\
    \label{eq:growth-frag-initial}
    g_0(x) = g_{\mathrm{in}}(x) \qquad (x > 0).
  \end{gather}
\end{subequations}
The unknown is a function $g_t(x)$ which depends on the time $t \geq
0$ and on $x > 0$, and for which an initial condition
$g_{\mathrm{in}}$ is given at time $t=0$. The positive function $\tau$
represents the \emph{growth rate}. The symbol $\LL$ stands for the
fragmentation operator (see below), and $\lambda$ is the largest
eigenvalue of the operator $g \mapsto - \partial_x (\tau g) + \LL
g$, acting on a function $g=g(x)$ depending only on $x$. The main
motivation for the study of equation \eqref{eq:growth-frag} is the
closely related
\begin{equation}
  \label{eq:growth-frag-orig}
  \partial_t n_t(x) + \partial_x (\tau(x) n_t(x))
  = \LL [n_t](x),
\end{equation}
with the same initial and boundary conditions. Solutions of the two
are related by $n_t(x) = e^{\lambda t} g_t(x)$, and $n_t$ represents
the size distribution at a given time $t$ of a population of cells (or
other objects) undergoing growth and division phenomena. The
population grows at an exponential rate determined by $\lambda > 0$,
called the \emph{Malthus parameter}, and approaches an asymptotic
shape for large times. Equation \eqref{eq:growth-frag} has a
stationary solution and is more convenient for studying its asymptotic
behavior, which is why it is commonly considered. Of course, results
about \eqref{eq:growth-frag} are easily translated to results about
\eqref{eq:growth-frag-orig} through the simple change $n_t(x) =
e^{\lambda t} g_t(x)$.

The \emph{fragmentation operator} $\LL$ acts on a function $g = g(x)$
as
\begin{equation*}
  \LL g(x) :=  \LL_+g(x)  - B(x) g(x),
\end{equation*}
where the positive part $\mathcal{L}_+$ is given by
\begin{equation*}
  \LL_+ g(x) :=  \int_x^\infty b(y,x) g(y) \,dy.
\end{equation*}
The coefficient $b(y,x)$, defined for $y \geq x \geq 0$, is the
\emph{fragmentation coefficient}, and $B(x)$ is the \emph{total
  fragmentation rate} of cells of size $x > 0$. It is obtained from
$b$ through
\begin{equation*}
  B(x) := \int_0^x \frac{y}{x} \, b(x,y) \,dy
  \qquad (x > 0).
\end{equation*}

The \emph{eigenproblem} associated to \eqref{eq:growth-frag} is the
problem of finding both a stationary solution and a stationary
solution of the dual equation, this is, the first eigenfunction of the
growth-fragmentation operator $g \mapsto -(\tau g)' + \LL(g)$ and of
its dual $\varphi \mapsto \tau \varphi' + \LL^*(\varphi)$. If
$\lambda$ is the largest eigenvalue of the operator $g \mapsto
- (\tau\, g)' + \LL g$, the associated eigenvector $G$
satisfies
\begin{subequations}
  \label{eq:G}
  \begin{gather}
    \label{eq:G-eq}
    (\tau(x)\,G(x))' + \lambda G(x)
    = \LL (G)(x),
    \\
    \label{eq:G-boundary}
    \tau(x)G(x)\big\vert_{x=0} = 0,
    \\
    \label{eq:G-positive-normalized}
    G \geq 0,
    \quad
    \int_0^\infty G(x)\,dx = 1.
  \end{gather}
\end{subequations}
Of course, the eigenvector $G$ is an equilibrium (i.e., a stationary
solution) of equation~\eqref{eq:growth-frag}. The associated dual
eigenproblem reads
\begin{subequations}
  \label{eq:phi}
  \begin{gather}
    \label{eq:phi-eq}
    - \tau(x) \phi'(x)
    + (B(x)+\lambda)\, \phi(x)
    = \LL_+^{*} \phi(x),
    \\
    \label{eq:phi-positive-normalized}
    \phi \geq 0,
    \quad
    \int_0^\infty G(x) \phi(x) \,dx = 1,
  \end{gather}
\end{subequations}
where
\begin{equation*}
  \LL_+^{*} \phi(x):=   \int_0^x b(x,y) \phi(y) \,dy.
\end{equation*}
This dual
eigenproblem is interesting because $\phi$ gives a conservation law
for \eqref{eq:growth-frag}:
\begin{equation*}
  \int_0^\infty \phi(x)\, g_t(x) \,dx
  =
  \int_0^\infty \phi(x)\, g_{\mathrm{in}}(x) \,dx
  = \text{Cst}
  \qquad (t \geq 0).
\end{equation*}
In this paper we always denote by $G$, $\phi$ and $\lambda$ the
solution to (\ref{eq:G}) and (\ref{eq:phi}).

In the rest of this introduction we describe the assumptions used
throughout the paper and state our main results. In Section
\ref{sec:G} we give the proof of our estimates on the stationary
solution $G$, and Section \ref{sec:phi} is devoted to estimates of the
dual eigenfunction $\phi$. Our results on the spectral gap of the
growth-fragmentation operator are proved in Section \ref{sec:entropy},
and we also include two appendices: one, Appendix
\ref{sec:app:approximation}, on different approximation procedures
that may be used for $G$ and $\phi$, and which are more convenient in
some of our arguments; and Appendix \ref{sec:Laplace}, which gives
asymptotic estimates of some of the expressions involving the positive
part $\LL^+$ of the fragmentation operator, and are used for our
large-$x$ estimates of $G$.

\subsection{Assumptions on the coefficients}

For proving our results we need some or all of the following
assumptions. First of all, we assume that the fragmentation
coefficient $b$ is of self-similar form, which is general enough to
encompass most interesting examples and still allows us to obtain
accurate results on the asymptotics of $G$ and $\phi$:

\begin{hyp}[Self-similar fragmentation rate]
  \label{hyp:self-similar-frag}
  The coefficient $b(x,y)$
  is of the form
  \begin{equation}
    \label{eq:b-ss}
    b(x,y) = B(x) \frac{1}{x}\,p\!\left( \frac{y}{x} \right)
  \end{equation}
  for some locally integrable $B:(0,+\infty) \to (0,+\infty)$, and
  some nonnegative finite measure $p$ on $[0,1]$ satisfying the mass preserving condition
  \begin{equation*}
    \int_0^1 z\, p(z) \,dz = 1,
  \end{equation*}
  and also the condition
  \begin{equation*}
    \int_0^1 p(z) \,dz > 1.
  \end{equation*}
  (When writing the integral of a measure it is always understood that
  the integration limits are included in the integral.)
\end{hyp}
The measure $p$ gives the distribution of fragments obtained when a
particle of a certain size breaks.

\begin{rem}\label{rem:b-moment-bounds}
  Define, for $k\geq0,$ the moment
  \begin{equation*}
    \pi_k:=\int_{0}^{1}z^kp(z)\,dz.
  \end{equation*}
  We have from Hypothesis~\ref{hyp:self-similar-frag} that $\pi_1=1$ and $\pi_0>1.$
  Physically, $\pi_0$ represents the mean quantity
  of fragments produced by the fragmentation of one particle.
  Because of the strict inequality $\pi_0>\pi_1,$ one can deduce that $p$ is not
  concentrated at $z=1$ ({\it i.e.} $p\neq \pi_0\delta_1$).
  As a consequence we have that $\pi_k<1$ if $k>1$ and $\pi_k>1$ if $k<1.$
\end{rem}

\begin{hyp}
  \label{hyp:tau}
  The growth rate is a continuous and strictly positive function
  $\tau:(0,+\infty) \to (0,+\infty)$.
\end{hyp}

Our next assumption says that the growth rate and total fragmentation
rate have a power-law behavior for large and small sizes:

\begin{hyp}[Asymptotics of fragmentation and drift rates]
  \label{hyp:asymptotics}
  Assume that for some constants $B_0, B_\infty, \tau_0, \tau_\infty >
  0$ and $\gamma_0, \gamma, \alpha_0, \alpha \in \R$
  \begin{gather}
    \label{eq:B-asymptotics-0}
    B(x) \sim B_0 \, x^{\gamma_0}
    \quad \text{as $x \to 0$},
    \\
    \label{eq:B-asymptotics-infty}
    B(x) \sim B_\infty \, x^{\gamma}
    \quad \text{as $x \to +\infty$},
    \\
    \label{eq:tau-asymptotics-0}
    \tau(x) \sim \tau_0 \, x^{\alpha_0}
    \quad \text{as $x \to 0$},
    \\
    \label{eq:tau-asymptotics-infty}
    \tau(x) \sim \tau_\infty \, x^{\alpha}
    \quad \text{as $x \to +\infty$}.
  \end{gather}
  We also impose the conditions
  \begin{gather}
    \label{eq:gamma0-alpha0+1}
    \gamma_0-\alpha_0+1 > 0,
    \\
    \label{eq:gamma-alpha+1}
    \gamma-\alpha+1 > 0,
  \end{gather}
 to ensure the existence of a solution to the {\it eigenproblem} (see~\cite{MR2652618,MR2250122}).
\end{hyp}

Likewise, we impose that the distribution of small fragments behave
like a power law:

\begin{hyp}[Behavior of $p$ close to $0$]
  \label{hyp:p-at-borders}
  There exist $p_0 
 \geq 0$ and $\mu > 0$
 such that
  \begin{equation}
    \label{eq:p-at-0}
    p(z) = p_0 \,z^{\mu-1}  + o(z^{\mu-1}) \quad \text{as $z \to 0$},
  \end{equation}
 with the condition
\begin{equation}\label{eq:mu-alpha0+1}
    \mu - \alpha_0 + 1 > 0.
  \end{equation}

\end{hyp}

\begin{rem}\label{rk:p0}
  When $p_0 > 0$ condition (\ref{eq:p-at-0}) is the same as
  \begin{equation}
    p(z) \sim p_0 \,z^{\mu-1}\label{eq:p-at-0-again}
  \end{equation}
  as $z \to 0$. We prefer to write it as given in order to allow for
  $p_0 = 0$, which is usually not allowed in the notation
  (\ref{eq:p-at-0-again}). For instance, if $p(z)$ is equal to $0$ in
  a neighborhood of $0$ (such as for the mitosis case, see below),
  then (\ref{eq:p-at-0}) holds with $p_0=0$, but
 (\ref{eq:p-at-0-again}) does not make sense.
\end{rem}

To find the asymptotic behavior of the function $G$ when $x\to\infty$
the following hypothesis will also be needed.
\begin{hyp}[Asymptotics to second order]
  \label{hyp:asymptotics-2nd}
  Assume that $\tau$ is a $\mathcal{C}^1$ function and that, for some
  $\delta > 0$ and $\nu > -1$,
  \begin{gather}
    \label{eq:B-asymptotics-infty-2nd}
    B(x) = B_\infty \, x^{\gamma} + O(x^{\gamma - \delta})
    \quad \text{as $x \to +\infty$},
    \\
    \label{eq:tau-asymptotics-infty-2nd}
    \tau(x) = \tau_\infty \, x^{\alpha} + O(x^{\alpha - \delta})
    \quad \text{as $x \to +\infty$},
    \\
    \label{eq:p-asymptotics-2nd}
    p(z) = p_1(1-z)^{\nu} + O((z-1)^{\nu+\delta})
    \quad \text{as $z \to 1$}.
  \end{gather}
\end{hyp}

Finally, to prove the entropy-entropy dissipation inequality, we will
need an additional restriction on the fragmentation coefficient. It
essentially says that $p$ is uniformly bounded below by some constant
$\underline p > 0$, and that it behaves like a constant at the
endpoints $0$ and $1$:
\vspace*{-5pt}
\begin{hyp}
  \label{hyp:p-bounded}
  There exist positive constants $\underline p,\,p_0,\,p_1>0$ such that
  \begin{equation*}
    \forall\,z\in(0,1),\quad p(z) \geq \underline p\qquad \text{(in the sense of measures)},
  \end{equation*}
\[\quad p(z)\xrightarrow[z\to0]{}p_0,\qquad p(z)\xrightarrow[z\to1]{}p_1,\]
and $\alpha_0<2$ (which is nothing but
condition~\eqref{eq:mu-alpha0+1} in the case $\mu=1$).
\end{hyp}

The reader may check that \cite[Theorem 1]{MR2652618}, which gives
existence and uniqueness of $G$, $\phi$, $\lambda$ satisfying
(\ref{eq:G}) and (\ref{eq:phi}) is applicable under Hypotheses
\ref{hyp:self-similar-frag}--\ref{hyp:p-at-borders}. We assume at
least these hypotheses throughout the paper in order to ensure the
existence of profiles.

Let us give some common examples of coefficients satisfying the above
assumptions:
\begin{description}
\item[Power coefficients] If we set
  \begin{equation*}
    b(x,y) = 2 x^{\gamma-1} \text{ for } x>y>0,
    \quad
    \tau(x) = x^\alpha \text{ for } x>0,
  \end{equation*}
  then all our hypotheses are satisfied when $\gamma - \alpha + 1 >
  0$ and $\alpha<2$. Observe that in this case $B(x) = x^\gamma$ and $p(z) \equiv 1$,
  which satisfies Hypotheses~\ref{hyp:self-similar-frag},
  \ref{hyp:p-at-borders} with $\mu = 1$ and $\nu = 0$, and also
  \ref{hyp:p-bounded}. Since $\tau(x)$ is a power, it satisfies
  Hypothesis~\ref{hyp:tau}.
  Hypotheses~\ref{hyp:asymptotics} and \ref{hyp:asymptotics-2nd} are
  also satisfied.
\item[Self-similar fragmentation] The previous case with $\tau(x)=x$
  is referred to as the \emph{self-similar fragmentation equation}. It
  is closely related to the fragmentation equation $\partial_t g_t =
  \LL(g_t)$ (see \cite{MR2114413,MR2832638}).
\item[Mitosis] Cellular division by equal mitosis is modeled by a
  distribution of fragments $p$ concentrated at a size equal to one
  half:
  \begin{equation*}
    p(z) = 2\,\delta_{z=\frac{1}{2}}.
  \end{equation*}
  This measure $p$ satisfies Hypothesis~\ref{hyp:p-at-borders} with
  $p_0 = p_1 = 0$ (the value of $\mu$, $\nu$ being irrelevant). In
  order to make the theory work, one has to choose $B$ and $\tau$ such
  that the rest of Hypotheses are satisfied. For instance, $B(x) =
  x^\gamma$ and $\tau(x) = x^\alpha$ with $\gamma - \alpha + 1 > 0$
  (and then defining $b(x,y)$ through \eqref{eq:b-ss}) are valid
  choices for the same reasons as before.

\end{description}
\vspace*{-10pt}
\subsection{Summary of main results}

\paragraph{Estimates on the profiles.}

We describe the asymptotics of the profile $G$ and give accurate
bounds on the eigenvector $\phi$. Define
\begin{equation*}
  \Lambda(x):=\int_1^x\frac{\lambda+B(y)}{\tau(y)}dy
\end{equation*}
and
\begin{equation*}
  \xi :=
  \begin{cases}
    p_1
    &\text{ if $\gamma > 0$ and $\nu = 0,$}
    \\
    p_1\frac{B_\infty}{\lambda+B_\infty}
    &\text{ if $\gamma=\nu=0$},
    \\
    0
    &\text{ if $\gamma \geq 0$ and $\nu > 0$, or $\gamma<0$ and $\nu>-1+\frac{\gamma+1-\alpha}{1-\alpha},$}
  \end{cases}
\end{equation*}
where the parameters are the ones appearing in the previous
hypotheses.  In Section~\ref{sec:est-G-infty} we prove the following
result, which improves previous estimates of the profile $G$ given in
\cite{MR2114413,MR2832638,MR2652618}
\begin{thm}
  \label{thm:G-at-infinity}
  Assume Hypotheses
  \ref{hyp:self-similar-frag}--\ref{hyp:asymptotics-2nd}. There exists
  $C>0$ such that
  \begin{equation}
    \label{eq:G-asymptotic-x-large}
    G(x) \underset{x \to +\infty}{\sim} C e^{-\Lambda(x)} x^{\xi-\alpha}.
  \end{equation}
\end{thm}
This result works for all the examples given before. For all of them,
it shows that the profile $G$ decays exponentially for large sizes,
with a precise exponential rate given by $\Lambda(x)$. We observe that
$\Lambda(x)$ behaves like $x^{\gamma_+ - \alpha + 1}$ (with $\gamma_+
:= \max\{\gamma,0\}$), which is always a positive power of $x$. There
are some observations about this that match intuition: the equilibrium
profile decays faster when the total fragmentation rate is stronger
for large sizes, and it decays slower when the growth rate is larger
for large sizes. Also, it is interesting to notice that $\Lambda$ does
not depend on the fragment distribution (this is, $p$), but only on
the total fragmentation rate $B$.

The additional power $x^{\xi-\alpha}$ which gives a correction to the
exponential behavior, in turn, depends only on the behavior of the
distribution of fragments $p(z)$ close to $z=1$, this is, on
fragments of size close to the size of the particle that breaks. In
the mitosis case, for example, $\xi = 0$ since we obtain no fragments
of similar size when a particle breaks.

\medskip The behavior of $G(x)$ for $x$ close to $0$ depends on the
power $\alpha_0$ from Hypothesis~\ref{hyp:asymptotics} and the
distribution of small fragments that result when a particle
breaks. The following result is proven in Section~\ref{sec:G-at-zero}:

\begin{thm}
  \label{thm:G-at-zero}
  Assume Hypotheses
  \ref{hyp:self-similar-frag}--\ref{hyp:p-at-borders} with $p_0>0.$
  If $\alpha_0<1$, there exists $C>0$ such that
  $$G(x)\underset{x\to0}{\sim}C\,x^{\mu-\alpha_0}.$$
  If $\alpha_0\geq 1,$ there exists $C>0$ such that
  $$G(x)\underset{x\to0}{\sim}C\,x^{\mu-1}.$$
\end{thm}

This shows that $G$ is (roughly) more concentrated close to $0$ the
weaker the growth is for smaller sizes; and is less concentrated when
there are fewer smaller fragments resulting from breakage. This result
includes cases in which $G(x)$ blows up as $x \to 0$, cases in which
it behaves like a constant, and cases in which it tends to $0$ like a
power. We recall that the boundary condition is $\tau(x) G(x) \to 0$
as $x \to 0$, which is always ensured by $\mu > 0$ from Hypothesis
\ref{hyp:p-at-borders}. 

\medskip
For the profile $\phi$ we derive the following estimates, proved in
section \ref{sec:phi}, by the use of a maximum principle
(Lemma~\ref{lem:maxprinciple}):

\begin{thm}
  \label{thm:allbounds}
  Assume Hypotheses~\ref{hyp:self-similar-frag}--\ref{hyp:p-at-borders}.
  If $\gamma>0$, there are two positive
  constants $C_1$ and $C_2$ such that
  \begin{equation}\label{eq:bounds-gamma-positive}
    C_1\, x\leq \phi(x)\leq C_2\, x,\qquad \forall x>1.
  \end{equation}
  If $\gamma<0$
  and under the additional assumption that $\mu=1$ and $p_0>0$ in Hypothesis~\ref{hyp:p-at-borders},
  there exist two positive constants $C_1$ and $C_2$ such that
  \begin{equation}\label{eq:bounds-gamma-negative}
    C_1\, x^{\gamma-1}
    \leq \phi(x)
    \leq C_2\, x^{\gamma-1},\qquad\forall x>1.
  \end{equation}
\end{thm}

Estimates of $\phi$ are significantly harder than those of $G$, and
they have to be obtained through comparison arguments. To our
knowledge, this is the first result in which $\phi$ can be bounded
above and below by the same power (except for the cases in which
$\phi$ can be found explicitly). This improves the results in
\cite{MR2832638} also in that it is valid for a general power-law behavior
of $\tau$.

We do not include the case $\gamma=0$ in the above theorem (this is,
$B(x)$ asymptotic to a constant as $x \to +\infty$), but we
remark that in the case of $B(x)$ \emph{equal} to a constant (and with
the very mild condition that $\int b(x,y) \,dy$ is equal to a constant
independent of $x$), then $\phi \equiv 1$. The case $\tau(x) = \tau_0
x$ is also explicit: in that case, $\lambda = \tau_0$ and $\phi(x) = C
x$ for some number $C > 0$.

\medskip
As for the behavior at zero, we prove the following result:
\begin{thm}
  \label{thm:phi-at-zero}
  Assume
  Hypotheses~\ref{hyp:self-similar-frag}--\ref{hyp:p-at-borders}. Then
  there exists a constant $C>0$ such that
  \begin{equation*}
    \phi(x)\underset{x\to0}{\sim} Ce^{\Lambda(x)}.
  \end{equation*}
\end{thm}
We remark that the behavior of $\Lambda(x)$ for small $x$ is
determined by whether $(B(x) + \lambda)/\tau(x)$ is integrable close
to $x=0$. Since $B(x)/\tau(x)$ is always integrable close to $x=0$ by
hypothesis (as $\gamma_0 - \alpha_0 > -1$), we deduce that:
\begin{enumerate}
\item If $\gamma_0 \leq 0$, then $\phi(x)$ tends to a positive constant
  as $x \to 0$.
\item If $\gamma_0 > 0$, then there are three possible cases:
  \begin{enumerate}
  \item If $\alpha_0 < 1$, then again $\phi(x)$ tends to a positive constant
    as $x \to 0$.
  \item If $\alpha_0 = 1$, then $\phi(x)$ behaves like a positive
    power of $x$ as $x \to 0$.
  \item If $\alpha_0 > 1$, then $\phi(x)$ decays exponentially fast as
    $x \to 0$.
  \end{enumerate}
\end{enumerate}

\paragraph{Spectral gap.}
The estimates of the previous theorems allow us to prove a spectral gap inequality.
The \emph{general relative
  entropy principle} \cite{MR2065377,MR2162224} applies here
and we have
\begin{multline*}
\frac{d}{dt}\int_{0}^{\infty}\phi(x)G(x)H(u(x))\,dx=\int_{0}^{\infty}\int_{y}^{\infty}\phi(y)b(x,y)	G(x)\\
\times \left( H(u(x))-H(u(y))+H'(u(x))(u(y)-u(x))\right)\,dx\,dy,
\end{multline*}
where $H$ is any function and we denote
\begin{equation*}
  u(x):=\frac{g(x)}{G(x)}\quad (x>0).
\end{equation*}
In the particular case of $H(x):=(x-1)^2$ we define
\begin{gather}
  \label{entropydefinition}
  H[g|G]:=\int_{0}^{\infty} \phi \, G(u-1)^2\,dx
  \\
  \label{entropydissipation}
  D[g|G]
  :=
  \int_{0}^{\infty}\int_{x}^{\infty}\phi(x)G(y)b(y,x)(u(x)-u(y))^2\,dy\,dx,
\end{gather}
and obtain that
\begin{equation*}
\frac{d}{dt}H[g|G]=-D[g|G]\leq 0.
\end{equation*}
The next result shows that $H$ is in fact bounded by a constant
times $D$:
\begin{thm}\label{thm:entropy_inequality}
  Assume that the coefficients satisfy
  Hypotheses~\ref{hyp:self-similar-frag}--\ref{hyp:p-bounded} with one
  of the following additional conditions on the exponents $\gamma_0$
  and $\alpha_0$:
\begin{itemize}
 \item either $\alpha_0>1$,
 \item or $\alpha_0=1$ and $\gamma_0\leq1+\lambda/\tau_0,$
 \item or $\alpha_0<1$ and $\gamma_0\leq2-\alpha_0.$
\end{itemize}
Consider also that we are in the case $\gamma\neq0.$
  Then the following inequality holds
  \begin{equation}\label{entropy1}
    H[g|G] \le C D[g|G],
  \end{equation}
  for some constant $C>0$ and for any nonnegative measurable function
  $g:(0,\infty)\to\mathbb{R}$ such that $\int \phi g=1$. Consequently,
  if $g_t$ is a solution of problem \eqref{eq:growth-frag} the speed
  of convergence to equilibrium is exponential in the $L^2$-weighted
  norm\\ \mbox{$\|\cdot\|=\|\cdot\|_{L^2(G^{-1}\phi dx)}$,} i.e.,
  \begin{equation*}
    H[g_t|G]
    \leq
    H[g_0|G] \, e^{-C t}
    \quad \text{ for $t \geq 0$}.
  \end{equation*}
\end{thm}

Remark that in general we do not know the value of the eigenvalue
$\lambda$ which appears in the assumption on $\gamma_0$ for the case
$\alpha_0=1.$ Nevertheless in the case of the self-similar
fragmentation equation (\emph{i.e.} $\tau(x)\equiv\tau_0x$) we know by
integration of equation~\eqref{eq:G-eq} multiplied by $x$ that
$\lambda = \tau_0$ and the condition on $\gamma_0$ becomes $\gamma_0
\leq 2.$ Thus Theorem~\ref{thm:entropy_inequality} includes the result
of the first part of \cite[Theorem~1.9]{MR2832638}.

The main restrictions on the coefficients needed for Theorem
\ref{thm:entropy_inequality} to hold are the following. First, we
require Hypothesis \ref{hyp:p-bounded}, which says that the fragment
distribution $p$ should be bounded below. Consequently, this does not
include the mitosis case and other cases in which the fragment
distribution has ``gaps''; we refer to \cite{MR2536450} for a proof
that exponential decay does hold in that case, at least for a constant
total fragmentation rate. Second, the exponent $\gamma_0$ cannot be
too large in order to ensure that the term $b(x,y)$ which appears in
the entropy dissipation is not too small and can be bounded below by
our methods. (An exception to this is the case $\alpha_0 > 1$: in this
case $\phi(x)$ decays exponentially fast as $x \to 0$, and this allows
us to remove the upper bound on $\gamma_0$.) This restriction on
$\gamma_0$ might be a shortcoming of the arguments we are using; we do
not know if there is a spectral gap when it is removed.

On the other hand, it is remarkable that Theorem
\ref{thm:entropy_inequality} does not place any restrictions on the
behavior of the fragmentation or growth coefficients for large
sizes. This is a significant improvement over \cite{MR2832638}, where
the behavior at $0$ and $+\infty$ of the coefficients was taken to be
the same power of $x$, and results were restricted to the cases in
which $\tau$ is constant or linear.

\section{Estimates of the profile $G$}
\label{sec:G}

\subsection{Estimates of the moments of $G$}
\label{sec:G-moments}

When Hypothesis~\ref{hyp:asymptotics} is satisfied, we define
\begin{equation}\label{eq:zeta}
  \zeta:=\lim_{x\to+\infty}x^{\alpha-\gamma_+}\frac{B(x)+\lambda}{\tau(x)}
  =
  \frac{\lambda\1_{\gamma\leq0}+B_\infty\1_{\gamma\geq0}}{\tau_\infty}
  =
  \begin{cases}
    \frac{B_\infty}{\tau_\infty} &\text{ if $\gamma > 0,$}
    \\
    \frac{\lambda+B_\infty}{\tau_\infty} &\text{ if $\gamma=0,$}
    \\
    \frac{\lambda}{\tau_\infty} &\text{ if $\gamma<0$,}
  \end{cases}
\end{equation}
where $\gamma_+=\max\{0,\gamma\}.$
Remark that, for $\gamma\geq0,$ we have the relation
\begin{equation}
  \label{eq:zeta_xi}\xi=p_1\frac{B_\infty}{\tau_\infty}\zeta^{-1}.
\end{equation}

\begin{lem}
  \label{lem:exp-moments-G}
  Assume Hypotheses
  \ref{hyp:self-similar-frag}--\ref{hyp:p-at-borders}. For any $m > 1
  + \xi$ it holds that
  \begin{equation*}
    \int_1^\infty G(x)\, e^{\Lambda(x)}\, x^{\alpha - m} \,dx
    < + \infty.
  \end{equation*}
\end{lem}

\begin{proof}
  As usual, we carry out a priori estimates which can be
  rigorously justified by an approximation procedure (such as the
  truncated equation \eqref{eq:truncated1}). As $G$ is integrable, it
  is enough to prove the convergence of the above integral on
  $(x_0,+\infty)$ for a sufficiently large $x_0 > 0$. Hence, take any
  $x_0 > 0$, multiply \eqref{eq:G-eq} by $x^{1-m} e^{\Lambda(x)}$ with
  $m > 1+\xi$ and integrate on $(x_0,+\infty)$ to obtain
  \begin{multline}
    \label{eq:moments-proof-1}
    - G(x_0) e^{\Lambda(x_0)} \tau(x_0) x_0^{1-m}
    + (m-1) \int_{x_0}^\infty G(x) e^{\Lambda(x)} \tau(x) x^{-m} \,dx
    \\
    =
    \int_{x_0}^\infty G(y)
    \int_{x_0}^y e^{\Lambda(x)} x^{1-m} b(y,x) \,dx \,dy
  \end{multline}
  where we have done an integration by parts on the last term.

  We first consider the case $\xi > 0$ (this is, $\gamma \geq 0$ and
  $\nu = 0$). From Equation~\eqref{eq:tau-asymptotics-infty} we have
  that for any $\epsilon > 0$ there exists $x_0 > 0$ such that
  \begin{multline}
    \label{eq:moments-proof-2}
    (m-1) \int_{x_0}^\infty G(x) e^{\Lambda(x)} \tau(x) x^{-m} \,dx
    \\
    \geq (m-1) (1-\epsilon) \tau_\infty \int_{x_0}^\infty G(x)
    e^{\Lambda(x)} x^{\alpha-m} \,dx,
  \end{multline}
  and, applying Lemma~\ref{lem:Laplace}, also such that
  \begin{multline}
    \label{eq:moments-proof-3}
    \int_{x_0}^\infty G(y) \int_{x_0}^y e^{\Lambda(x)} x^{1-m}
    b(y,x) \,dx \,dy
    \\
    \leq (1+\epsilon) B_\infty p_1\zeta^{-1}
    \int_{x_0}^\infty G(y) e^{\Lambda(y)} y^{\alpha-m}\,dy
  \end{multline}
  (observe that we have used $\gamma \geq 0$ and $\nu = 0$ here).
  Using (\ref{eq:moments-proof-2}) and (\ref{eq:moments-proof-3}) we
  obtain from (\ref{eq:moments-proof-1}) that
  \begin{multline*}
    \Big( (m-1) (1-\epsilon) \tau_\infty
    - (1+\epsilon) B_\infty p_1\zeta^{-1} \Big)
    \int_{x_0}^\infty G(x) e^{\Lambda(x)} x^{\alpha-m} \,dx
    \\
    \leq
    G(x_0) e^{\Lambda(x_0)} \tau(x_0) x_0^{1-m}.
  \end{multline*}
  When $(m-1) (1-\epsilon) \tau_\infty - (1+\epsilon) B_\infty p_1\zeta^{-1} >
  0$ this gives a bound for the integral on the left hand side. If $m
  > 1+\xi$ we can always choose $\epsilon$ small enough for this to
  be true, because of relation~\eqref{eq:zeta_xi}, and it proves the result.

  The remaining case is $\xi = 0$, this is, $\nu>-1+\frac{\gamma+1-\alpha}{\gamma_++1-\alpha}.$
  In this case we have to substitute (\ref{eq:moments-proof-3}) by
  the following, according to Lemma \ref{lem:Laplace}:
  \begin{multline}
    \label{eq:moments-proof-4}
    \int_{x_0}^\infty G(y)
    \int_{x_0}^y e^{\Lambda(x)} x^{1-m} b(y,x) \,dx \,dy
    \\
    \leq
    (1+\epsilon) B_\infty p_1 \zeta^{-1-\nu}\Gamma(1+\nu) \int_{x_0}^\infty G(y)
    e^{\Lambda(y)}
    y^{1-m+\gamma -(\gamma_+-\alpha+1)(1+\nu)}
    \,dy.
  \end{multline}
  Since
  $\nu>-1+\frac{\gamma+1-\alpha}{\gamma_++1-\alpha}$, we have
  \begin{equation*}
    1-m+\gamma -(\gamma_+-\alpha+1)(1+\nu) < -m+\alpha.
  \end{equation*}
  Thus the exponent of $y$ on the right hand side of
  (\ref{eq:moments-proof-4}) is strictly smaller than $\alpha - m$, so
  we can always find $x_0$ large enough so that
  \begin{equation*}
    \int_{x_0}^\infty G(y)
    \int_{x_0}^y e^{\Lambda(x)} x^{1-m} b(y,x) \,dx \,dy
    \\
    \leq
    \epsilon \int_{x_0}^\infty G(y)
    e^{\Lambda(y)}
    y^{\alpha - m}
    \,dy.
  \end{equation*}
  Using this and (\ref{eq:moments-proof-2}) in
  (\ref{eq:moments-proof-1}) we may follow a similar reasoning as
  before to obtain the result.
\end{proof}

\subsection{Asymptotic estimates of $G$ as $x \to +\infty$}
\label{sec:est-G-infty}

In this section we prove Theorem~\ref{thm:G-at-infinity} by using the
moment estimates in Section \ref{sec:G-moments}.

\begin{proof}[{\bf Proof of Theorem~\ref{thm:G-at-infinity}}]
  We divide the proof in two steps:

  \paragraph{Step 1: proof that the limit is finite.}
  Again, we carry out a priori estimates on the solution which can be
  fully justified by using the approximation
  \eqref{eq:truncated1}. Let us first prove that $x^{\alpha-\xi} G(x)
  e^{\Lambda(x)}$ has a finite limit $C \geq 0$ as $x \to +\infty$,
  and later we will show that $C > 0$. We use equation \eqref{eq:G-eq}
  to obtain
  \begin{multline*}
    (x^{-\xi} \tau(x) G(x) e^{\Lambda(x)})'
    = -\xi x^{-\xi-1} \tau(x) G(x) e^{\Lambda (x)}
    \\
    + x^{-\xi} e^{\Lambda (x)}
    \int_{x}^{\infty}
    b(y,x) G(y) \dy.
  \end{multline*}
  Let us show that the right hand side of this last expression is
  integrable on $(x_0,+\infty)$ for some $x_0 > 0$. Once we have this
  the result is proved, since then $x^{\alpha-\xi} G(x)
  e^{\Lambda(x)}$ must have a limit as $x \to +\infty$. Integrating
  the right hand side we obtain:
  \begin{multline}
    \label{eq:G-bound-proof1}
    -\xi \int_{x_0}^\infty x^{-\xi-1} \tau(x) G(x) e^{\Lambda (x)} \dx
    + \int_{x_0}^\infty x^{-\xi} e^{\Lambda (x)}
    \int_{x}^{\infty}
    b(y,x) G(y) \dy \dx
    \\
    =
    \int_{x_0}^\infty G(x) \left(
      \int_{x_0}^x
      y^{-\xi}
      b(x,y) e^{\Lambda(y)} \dy
      - \xi x^{-\xi-1} \tau(x) e^{\Lambda (x)}
    \right)
    \,dx.
  \end{multline}
  We just need to show that the parenthesis is of the order of
  $e^{\Lambda(x)} x^{\alpha-\xi-1-\epsilon}$ for some $\epsilon > 0$,
  and then Lemma \ref{lem:exp-moments-G} shows that the above integral
  is finite.

  \paragraph{The case $\xi > 0$.}

  Let us start considering the case $\xi > 0$ (this is, $\gamma \geq
  0$ and $\nu = 0$).
  Using Lemma \ref{lem:Laplace}
  \begin{equation}
    \label{eq:G-bound-proof2}
    \int_{x_0}^x
    y^{-\xi}
    b(x,y) e^{\Lambda(y)} \dy
    =
    p_1B_\infty\zeta^{-1} \, x^{\alpha-\xi-1}e^{\Lambda(x)}
    +
    O(x^{-\eta+\alpha-1-\epsilon})e^{\Lambda(x)},
  \end{equation}
  for some $\epsilon > 0$. From \eqref{eq:tau-asymptotics-infty-2nd}
  we also have
  \begin{equation}
    \label{eq:G-bound-proof2.5}
    \xi x^{-\xi-1} \tau(x) e^{\Lambda (x)}
    =
    \tau_\infty \,\xi\, x^{\alpha-\xi-1} e^{\Lambda (x)}
    + O(x^{\alpha-\xi-1-\delta}) e^{\Lambda (x)}.
  \end{equation}
  Using \eqref{eq:G-bound-proof2}-\eqref{eq:G-bound-proof2.5} and the relation~\eqref{eq:zeta_xi}, the
  parenthesis in \eqref{eq:G-bound-proof1} is, in absolute value, less
  than $C x^{\alpha-\xi-1-\delta} e^{\Lambda (x)}$ for some constant
  $C > 0$. Hence by Lemma \ref{lem:exp-moments-G} the integral in
  \eqref{eq:G-bound-proof1} is finite, and we conclude that $x^{\alpha
    - \xi} G(x) e^{\Lambda(x)}$ has a finite limit as $x \to +\infty$
  when $\xi > 0$.

 \paragraph{The case $\xi = 0$.} In this case we have from Lemma~\ref{lem:Laplace}
 \begin{equation*}
   \int_{x_0}^x
   b(x,y) e^{\Lambda(y)} \dy
   \sim
   \, p_1B_\infty\zeta^{-1-\nu}\Gamma(1+\nu)x^{\gamma -(\gamma_+-\alpha+1)(1+\nu)}e^{\Lambda(x)}.
 \end{equation*}
 Using the same reasoning as at the end of the proof of Lemma \ref{lem:exp-moments-G} we
 have that, when $\xi = 0$,
 \begin{equation*}
   \gamma -(\gamma_+-\alpha+1)(1+\nu)
   < \alpha - 1,
 \end{equation*}
 which then shows that the right hand side of
 (\ref{eq:G-bound-proof1}) is finite due to Lemma
 \ref{lem:exp-moments-G}.

 \paragraph{Step 2: proof that $C > 0$.}

 In order to show that $C > 0$ in (\ref{eq:G-asymptotic-x-large}) set
 $F(x) := \tau(x) G(x) e^{\Lambda(x)}$ and obtain the following from
 (\ref{eq:G-eq}):
 \begin{equation}
   \label{eq:F-derivative}
   F'(x) =
   e^{\Lambda (x)}
   \int_{x}^{\infty}
   b(y,x) G(y) \dy.
 \end{equation}
 In particular, $F$ is nondecreasing, and this is enough to conclude
 in the case $\xi = 0$ (since then $\tau(x) G(x) e^{\Lambda(x)}$ must
 converge to a positive quantity, so the same must be true of
 $x^\alpha G(x) e^{\Lambda(x)}$). In the case $\xi > 0$ we may bound
  \begin{multline*}
    F'(x)
    =
    e^{\Lambda (x)}
    \int_{x}^{\infty}
    b(y,x) \frac{1}{\tau(y)} e^{-\Lambda(y)} F(y) \dy
    \\
    \geq
     F(x) e^{\Lambda (x)}
    \int_{x}^{\infty}
    b(y,x) \frac{1}{\tau(y)} e^{-\Lambda(y)} \dy,
  \end{multline*}
  which implies that
  \begin{equation*}
    F(x) \geq F(x_0) \exp \left(
      \int_{x_0}^x S(w) \d w,
    \right)
  \end{equation*}
  with
  \begin{equation*}
    S(w) :=
    e^{\Lambda (w)}
    \int_{w}^{\infty}
    b(y,w) \frac{1}{\tau(y)} e^{-\Lambda(y)} \dy.
  \end{equation*}
  Due to equation (\ref{eq:tau-asymptotics-infty-2nd}) we have
  \begin{equation*}
    \frac{1}{\tau(x)}
    = \frac{1}{\tau_\infty x^\alpha} + R_1(x),
  \end{equation*}
  with $R_1(x) = O(x^{-\alpha-\delta})$. Using this, and due to Lemma
  \ref{lem:Laplace},
  \begin{multline*}
    \int_{x_0}^x S(w) \d w
    \geq
    \int_{x_0}^x e^{\Lambda (w)}
    \int_{w}^{x}
    b(y,w) \frac{1}{\tau(y)} e^{-\Lambda(y)} \dy \d w
    \\
    =
    \int_{x_0}^x \frac{1}{\tau(y)} e^{-\Lambda(y)}
    \int_{x_0}^{y} e^{\Lambda (w)} b(y,w) \d w \d y
    \\
    =
    p_1 B_\infty\zeta^{-1} \int_{x_0}^x \frac{1}{\tau(y)}
    (y^{\alpha - 1} + R_2(y))
    \d y
    \\
    =
    \xi \int_{x_0}^x \frac{1}{y}
    \d y
    +
    \int_{x_0}^x
    R_3(y)
    \d y
    \geq
    \xi \log(y) + C_1,
  \end{multline*}
  with $R_2(y) = O(y^{\alpha - 1 -\epsilon})$, $R_3(y) = O(y^{-1-\epsilon})$,
  and $C_1 \in \R$ some real number. As a consequence,
  \begin{equation*}
    F(x) \geq F(x_0) x^\xi e^{C_1},
  \end{equation*}
  which shows that $\lim_{x \to +\infty} F(x) x^{-\xi}$ (which we
  know exists) must be strictly positive. This finishes the proof.
\end{proof}

\subsection{Asymptotic estimates of $G$ as $x \to 0$}
\label{sec:G-at-zero}

\begin{proof}[{\bf Proof of Theorem~\ref{thm:G-at-zero}}]
Define
$$F(x):=\tau(x)G(x)e^{\Lambda(x)}.$$
We know from~\cite{MR2652618} that $F(x)\to0$ when $x\to0$ and more
precisely that $F(x)\leq C\,x^\mu.$ The derivative of $F$, as noted in
(\ref{eq:F-derivative}), is
$$F'(x)=e^{\Lambda(x)}\int_x^\infty b(y,x)G(y) \dy>0$$
so $F$ is increasing.

\paragraph{Case $\alpha_0<1$}

In this case, $\Lambda(x)\to\Lambda(0)<0$. Choose $\epsilon > 0$ such
that $p$ is a function on $[0,\epsilon)$ (the fact that this can be
done for small enough $\epsilon$ is implicit in
Hypothesis~\ref{hyp:p-at-borders}), and call $p_* = p\,
\1_{[0,\epsilon]}$. Then, from
Hypothesis~\ref{hyp:p-at-borders},
\begin{equation}\label{eq:equiv_p_zero}
  x^{1-\mu} p_*\Big( \frac{x}{y} \Big) \to p_0 \, y^{1-\mu}
  \quad \text{ as $x \to 0$},
\end{equation}
with the above convergence being pointwise in $y$. We may additionally
choose $\epsilon \in (0,1)$ and $C > 0$ such that
\begin{equation}
  \label{eq:bound_p_zero}
  p(z)\leq Cz^{\mu-1}
  \fa z\,\in(0,\epsilon).
\end{equation}
Now we write
\begin{multline*}
  x^{1-\mu} \int_x^\infty b(y,x)G(y)\,dy
  \\
  = x^{1-\mu} \int_x^{\frac x\epsilon}
  \frac{B(y)}{y}G(y)p\Bigl(\frac xy\Bigr) \,dy
  +x^{1-\mu} \int_{\frac x\epsilon}^\infty
  \frac{B(y)}{y}G(y)p\Bigl(\frac xy\Bigr)\,dy
  \\
  = x^{1-\mu} \int_\epsilon^1
  B\Bigl(\frac xz\Bigr)G\Bigl(\frac xz\Bigr) p(z)\frac{dz}z
  + x^{1-\mu} \int_{\frac x\epsilon}^\infty
  \frac{B(y)}{y}G(y)p_*\Bigl(\frac xy\Bigr)\,dy.
\end{multline*}
For the first term in the r.h.s., we use that
$B(y)\underset{y\to0}{\sim}B_0 y^{\gamma_0}$ and $G(y) \leq C y^{\mu -
  \alpha_0}$ (see~\cite{MR2652618}) to write
\begin{equation*}
  x^{1-\mu} \int_\epsilon^1B
  \Bigl(\frac xz\Bigr)G\Bigl(\frac xz \Bigr) p(z) \frac{dz}z
  \leq
  C x^{\gamma_0+1-\alpha_0} \int_\epsilon^1 z^{\alpha_0-\mu-\gamma_0-1}p(z)\,dz
\end{equation*}
and conclude that it tends to zero when $x\to0$ since
$\gamma_0+1-\alpha_0>0.$ For the second term, we
use~\eqref{eq:equiv_p_zero} and \eqref{eq:bound_p_zero} to obtain by
dominated convergence
\begin{equation*}
  x^{1-\mu} \int_{\frac x\epsilon}^\infty
  \frac{B(y)}{y}G(y)p_*\Bigl(\frac xy\Bigr) \,dy
  \xrightarrow[x\to0]{}
  p_0 \int_0^\infty B(y) y^{-\mu} G(y) \,dy.
\end{equation*}
This limit is strictly positive and finite, since $G(y)\leq C y^{\mu -
  \alpha_0}$ and $\gamma_0-\alpha_0>-1.$ Finally, we have deduced that
there is a positive constant $C>0$ such that
\begin{equation*}
  F'(x) \underset{x \to 0}{\sim} C\,x^{\mu-1},
\end{equation*}
which by integration gives
$$\tau(x)G(x)\sim C\,x^\mu$$
and so
\begin{equation*}
  G(x)
  \underset{x \to 0}{\sim}
  C\,\frac{x^\mu}{\tau(x)}
  \underset{x \to 0}{\sim}
  C\,x^{\mu-\alpha_0}.
\end{equation*}

\paragraph{Case $\alpha_0 \geq 1$}

In this case we necessarily have $\gamma>0$ and
$$\Lambda(x)\sim -C\,x^{1-\alpha_0}.$$
As a consequence, following a similar reasoning as for the previous
case, we have
$$F'(x)\sim C_1\,x^{\mu-1}e^{-C_2\,x^{1-\alpha_0}}$$
and consequently
$$F(x)\sim C_1\int_0^x y^{\mu-1}e^{-C_2\,y^{1-\alpha_0}}dy\sim C_3\,x^{\alpha_0+\mu-1}e^{-C_2\,x^{1-\alpha_0}}$$
due to l'H\^opital's rule.  This finally gives $\displaystyle G(x)
\underset{x \to 0}{\sim} C\,x^{\mu-1}.$
\end{proof}

\vspace*{5pt}
\section{Estimates of the dual eigenfunction $\phi$}
\label{sec:phi}

\subsection{Asymptotic estimates of $\phi$ as $x\to0$}

We first give the proof of Theorem~\ref{thm:phi-at-zero}, which is
rather direct:

\begin{proof}[{\bf Proof of Theorem~\ref{thm:phi-at-zero}}]
  Define
  $$\psi(x):=\phi(x)e^{-\Lambda(x)}.$$
  This function is decreasing since it satisfies
  \begin{equation*}
    \psi'(x)=
    -\frac{1}{\tau(x)}\int_0^xb(x,y)\phi(y)\,dy\, e^{-\Lambda(x)}<0.
  \end{equation*}
  Moreover it is a positive function, so to prove
  Theorem~\ref{thm:phi-at-zero} we only have to prove that $\psi$ is
  bounded at $x=0.$ Consider, for $\eta>0,$ $\tau_\eta$ as defined in
  the approximation procedure (see \eqref{eq:tau_eta} in
  Appendix~\ref{sec:app:approximation}).  Then denote by $\phi_\eta,$
  $\Lambda_\eta$ and $\psi_\eta$ the corresponding functions.  First
  we know from~\cite{MR2652618} that $\phi_\eta$ converges locally
  uniformly to $\phi$ when $\eta\to0.$ We have, for $\eta>0,$ that
  $-\Lambda_\eta(x)=\int_x^1\frac{\lambda+B(y)}{\tau_\eta(y)}dy$ is
  bounded at $x=0$ and this is why it is useful to consider this
  regularization.  We have for any $x_0>0,$
  \begin{align*}
    \sup_{\R^+}\psi_\eta=\psi_\eta(0)
    &=\psi_\eta(x_0)+\int_0^{x_0}\frac{1}{\tau_\eta(y)}\int_0^yb(y,z)\phi_\eta(z)\,dz\,e^{-\Lambda_\eta(y)}\,dy\\
    &\leq\psi_\eta(x_0)+\int_0^{x_0}\frac{1}{\tau(y)}\int_0^yb(y,z)\phi_\eta(z)\,e^{-\Lambda_\eta(z)}\,dz\,dy\\
    &\leq\psi_\eta(x_0)+\sup\psi_\eta\int_0^{x_0}\frac{1}{\tau_\eta(y)}\int_{0}^{y}b(y,z)\,dz\,dy\\
    &=\psi_\eta(x_0)+\sup\psi_\eta\int_0^{x_0}\frac{B(y)}{\tau_\eta(y)}\int_{0}^{y}p\left(\frac{z}{y}\right)\,\frac{dz}y\,dy\\
    &=\psi_\eta(x_0)+\pi_0\sup\psi_\eta\int_0^{x_0}\frac{B(y)}{\tau(y)}\,dy.\\
  \end{align*}
  Now, because $\frac{B}{\tau}$ is integrable at $x=0,$ we can choose
  $x_0>0$ such that $\pi_0\int_0^{x_0}\frac{B(y)}{\tau(y)}\,dy\\ =\rho<1$
  and we obtain
  $$(1-\rho)\sup\psi_\eta(x)\leq\psi_\eta(x_0)\xrightarrow[\eta\to0]{}\psi(x_0).$$
  So $\psi_\eta$ is uniformly bounded when $\eta\to0$ and thus the
  limit $\psi(x)$ is bounded.
\end{proof}

\subsection{A maximum principle}

For finding the bounds on the dual eigenfunction at
$x\to+\infty$ we use comparison arguments, valid for each truncated
problem on $[0,L]$ (see Appendix~\ref{sec:app:approximation} for
details on the truncation). Then we pass to the limit, as the bounds
we obtain are independent of $L$. 
The function $\phi_L(x)$ satisfies the equation
\[
  \mathcal{S}\phi_L(x)=0 \quad (x\in(0,L)),
\]
where $\mathcal{S}$ is the operator given by
\begin{equation*}
\mathcal{S}w(x):=-\tau(x)w'(x)+(B(x)+\lambda_L)w(x)-\int_{0}^{x}b(x,y)w(y)dy,
\end{equation*}
defined for all functions $w\in W^{1,\infty}(0,L)$ and for $x\in(0,L)$.
This operator satisfies
\begin{equation}\label{eq:duality}\forall w\in W^{1,\infty}(0,L)\ s.t.\ w(L)=0,\qquad \int_0^L\mathcal S w(x)\,G_L(x)\,dx=0\end{equation}
where $G_L$ is the eigenfunction of the truncated growth-fragmentation
operator.  We recall the concept of \emph{supersolution}:
\begin{dfn}
We say that $w\in W^{1,\infty}(0,L)$ is a \emph{supersolution} of $\mathcal{S}$ on the interval $I\subseteq (0,L)$ when
\[
    \mathcal{S}w(x)> 0\ (x\in I).
\]
\end{dfn}
Maximum principles were a powerful tool for proving the existence of sub and supersolutions for the growth-fragmentation models as in \cite{MR2250122,MR2652618}.
For our case, we recall the maximum principle given in \cite{MR2652618}.
\begin{lem}[Maximum principle for $\mathcal{S}$]
  \label{lem:maxprinciple} Assume
  Hypotheses~\ref{hyp:self-similar-frag}-\ref{hyp:asymptotics}. There
  exists $A>0$, independent of $L$, such that if $ w$ is a
  supersolution of $\mathcal{S}$ on $(A,L)$, $w\geq 0$ on $[0,A]$ and
  $ w(L)\geq 0$ then $ w\geq 0$ on $[A,L]$.
\end{lem}
\begin{proof}
We start from the fact $w$ is a supersolution on $(A,L)$
\[-\tau(x) w'(x)+(B(x)+\lambda_L) w(x)-\int_0^x b(x,y) w(y)\,dy=:f(x)>0.\]
Testing this equation against $\1_{w\leq0}$ we obtain on $(A,L)$
\begin{align*}
-\tau(x) w_-'(x)+(B(x)+\lambda_L) w_-(x)&=\1_{ w(x)<0}\int_0^x b(x,y) w(y)\,dy+f(x)\1_{w(x)\leq0}\\
&\geq  \int_0^x b(x,y) w_-(y)\,dy+f(x)\1_{w(x)\leq0}.
\end{align*}
Extend $f$ by zero on $[0,A].$
Since $w_-(x)=0$ on $[0,A]$ by assumption, the latter inequality holds true on $(0,L)$ and it writes
\[\forall x\in(0,L),\quad\mathcal S w_-(x)\geq f(x)\1_{w(x)\leq0}.\]
Testing this last inequality against $G_L$, we obtain using~\eqref{eq:duality}
\[0\geq\int_0^Lf(x)\1_{w(x)\leq0}G_L(x)\,dx=\int_A^Lf(x)\1_{w(x)\leq0}G_L(x)\,dx.\]
Because $f$ and $G_L$ are positive on $(A,L)$, this is possible only if $\1_{w\leq0}\equiv0$ on $(A,L)$ and it ends the proof.
\end{proof}

\subsection{Asymptotic estimates of $\phi$ as $x\to+\infty$}

Now we prove the results concerning the asymptotic behavior
of $\phi(x)$ when $x\to+\infty$, Theorem~\ref{thm:allbounds}.  For
these results, we still assume that
Hypotheses~\ref{hyp:self-similar-frag}-\ref{hyp:p-at-borders} are
satisfied and, in the case $\gamma<0,$ we additionally assume that $\mu=1$ and $p_0>0$ (so that $p(z)\xrightarrow[z\to0]{}p_0>0).$
We recall that Hypothesis~\ref{hyp:asymptotics} says that $B(x)$
behaves like a $\gamma$-power of $x$ and $\tau(x)$ like an
$\alpha$-power of $x,$ with $\gamma+1-\alpha>0.$

\begin{proof}[{\bf Proof of Theorem~\ref{thm:allbounds}}]
  The proof is done in two cases, and each case is proved in two
  steps. In the first step we give particular supersolutions and prove
  the upper bound, and in the second one we do the corresponding for
  lower bounds.

  \medskip
  \noindent
  {\bf Case 1: $\gamma >0$.}

  {\it Step 1: Upper bounds.} We claim that for any $C>0,$ there
  exists $A>0$ and $L_*>0$ such that $$v(x):=Cx+1-x^k$$ is a
  supersolution on $[A,L]$ for any $L>L_*,$ provided that
  $\max(0,\alpha-\gamma)<k<1.$ First we recall that
  $\gamma+1-\alpha>0$ by assumption, so $\alpha-\gamma<1$ and we can
  find $k\in(\alpha-\gamma,1).$ Then
  $$\mathcal{S}v(x)=-\tau(x)(C-kx^{k-1})+\lambda(C x-x^k+1)+(\pi_k-1)B(x)x^k-(\pi_0-1)B(x)$$
  and the right hand side is positive for $x$ large enough because the
  dominant term is $C\lambda x+(\pi_k-1)B(x)x^k\sim C\lambda
  x+(\pi_k-1)B_\infty x^{\gamma+k}.$ Indeed $\pi_k>1$ because $k<1$
  (see Remark~\ref{rem:b-moment-bounds}) and the dominant power is
  $\gamma+k$ because $k>0$ and $\gamma+k>\alpha.$

  \

  Now we prove that there exists $C>0$ such that
  $$\forall\,x>0,\qquad\phi(x)\leq C(1+x).$$
  First we can choose $C$ such that $v(x)=Cx+1-x^k$ is bounded below
  by a positive constant. Moreover we take an approximation $\phi_L$
  of $\phi$ such that $\phi_L(L)=0$.  Then, choosing $K>0$ large
  enough, we have that $Kv(x)>\phi(x)$ on $[0,A]$ because $\phi$ is
  bounded uniformly in $L$ on $[0,A]$, and $Kv(L)=KCL+K-KL^k>0$ for
  $L$ large enough.  So, using the maximum principle and the previous
  lemma, we obtain that
  $$\forall\,x>0,\qquad\phi(x)\leq Kv(x)\leq C(1+x).$$

  \

  {\it Step 2: Lower bounds.} For the lower bounds we first prove that $v(x):=x+x^k-1$ is a subsolution for $\max(0,1-\gamma)<k<1.$\\
  The idea is to use $x^k$ to transform $x$ which is a supersolution
  into a subsolution.
  $$\mathcal{S}v(x)=-\tau(x)(1+kx^{k-1})+\lambda(x+x^k-1)-(\pi_k-1)B(x)x^k+(\pi_0-1)B(x)$$
  where $\pi_k>1$ since $k<1.$
  Due to Assumption~\eqref{eq:B-asymptotics-infty}, $B(x)x^k\sim B_\infty x^{\gamma+k}$ and $v(x)$ is a subsolution for $x$ large because $k>0$ and $\gamma+k>1.$\\
  \

  For $\gamma>0,$ there exists $C>0$ such that
  $$\forall\,x>0,\qquad \phi(x)\geq C(x-1)_+.$$
  We know that $\phi$ is positive, so for $C$ small enough,
  $C(x+x^k-1)-\phi(x)<0$ on $[0,A]$.  Moreover, taking an
  approximation $\phi_L$ of $\phi$ such that $\phi_L(L)= L$, we have
  $Cv(L)-\phi(L)<0$ for $C<1$ and $L$ large enough.  Finally we use
  the lemmas on the maximum principle and the subsolution to conclude
  that there exists $C>0$ such that
  $$\forall x>0,\qquad\phi(x)>C(x+x^k-1)$$
  and the result follows.\\
  \strut\\
  {\bf Case 2: $\gamma <0$.}

{\it Step 1: Upper bounds.} We start by proving that for any $\eta>\bigl(\frac{-\gamma\lambda}{B_\infty p_0}\bigr)^{\frac1\gamma},$ $v(x)=(\eta+x)^{\gamma-1}$ is a supersolution.
We compute
\begin{align*}
\mathcal S v(x)&=(1-\gamma)\tau(x)(\eta+x)^{\gamma-2}+(\lambda+B(x))(\eta+x)^{\gamma-1}\\
&\hspace{6cm}-\int_0^x b(x,y)(\eta+y)^{\gamma-1}\,dy
\end{align*}
and to estimate the last term in the r.h.s. we proceed similarly as in the proof of Theorem~\ref{thm:G-at-zero}.
We write, for $\epsilon\in(0,1),$
\begin{align*}
\int_0^x b(x,y)(\eta+y)^{\gamma-1}\,dy = \frac{B(x)}{x}&\int_0^{\epsilon x}(\eta+y)^{\gamma-1}p\Bigl(\frac{y}{x}\Bigr)\,dy\\
& +B(x)\int_\epsilon^1(\eta+zx)^{\gamma-1}p(z)\,dz.
\end{align*}
Then, choosing $\epsilon$ such that~\eqref{eq:bound_p_zero} is satisfied (for this we use Hypothesis~\ref{hyp:p-at-borders}), we obtain by dominated convergence from~\eqref{eq:equiv_p_zero} that
\[\frac{B(x)}{x}\int_0^{\epsilon x}(\eta+y)^{\gamma-1}p\Bigl(\frac{y}{x}\Bigr)\,dy\underset{x\to+\infty}{\sim}\frac{B(x)}{x}\frac{p_0\eta^\gamma}{-\gamma}.\]
On the other hand we have
\[B(x)\int_\epsilon^1(\eta+zx)^{\gamma-1}p(z)\,dz\underset{x\to+\infty}{\sim}x^{\gamma-1}B(x)\int_\epsilon^1 z^{\gamma-1}p(z)\,dz.\]
Since $\gamma<0,$ we obtain
\[\int_0^x b(x,y)(\eta+y)^{\gamma-1}\,dy\underset{x\to+\infty}{\sim}\frac{B_\infty p_0\eta^\gamma}{-\gamma}x^{\gamma-1}\]
and finally
\[\mathcal S v(x)\underset{x\to+\infty}{\sim}\left(\lambda-\frac{B_\infty p_0\eta^\gamma}{-\gamma}\right)x^{\gamma-1}\]
because $\tau(x)\sim\tau_\infty x^{\alpha}$ and $\alpha-1<\gamma<0.$
So $v(x)$ is a supersolution for $x$ large when $\eta>\bigl(\frac{-\gamma\lambda}{B_\infty p_0}\bigr)^{\frac1\gamma}.$\\

Now, we claim that there exist $C>0$ and $\epsilon>0$ such that
$$\forall\,x>0,\qquad \phi(x)\leq C(\eta+x)^{\gamma-1}.$$
The proof of this fact follows from the  maximum principle and taking the an approximation $\phi_L$ of $\phi$ such that $\phi_L(L)=0$ and that $v(x)$ is a supersolution.\\

{\it Step 2: Lower bounds.} For the lower bounds we define
\begin{equation*}
v(x):=
\left\{\begin{array}{ll}
0&\quad\text{for}\ 0<x<\epsilon,\\
(x-\epsilon)x^{\gamma-2}&\quad\text{for}\ x>\epsilon.
\end{array}\right.
\end{equation*}
Then for $\epsilon<\left(\frac{\lambda\gamma(\gamma-1)}{B_\infty p_0}\right)^{\frac1\gamma},$ $v$ is a subsolution.
Indeed we have for $x>\epsilon$
\begin{align*}
\mathcal S v(x)&=\tau(x)(x^{\gamma-2}+(\gamma-2)(x-\epsilon)x^{\gamma-3})+(\lambda+B(x))(x-\epsilon)x^{\gamma-2}\\
&\hspace{3cm}-\int_\epsilon^x b(x,y)(y-\epsilon)y^{\gamma-2}\,dy
\end{align*}
and, reasoning as in Step~1, we obtain that
\[\mathcal S v(x)\underset{x\to+\infty}{\sim}(\lambda-B_\infty p_0 C_\epsilon)x^{\gamma-1}.\]

Finally, there exist $C>0$ and $\epsilon>0$ such that
$$\forall\,x>0,\qquad\phi(x)\geq Cx^{\gamma-2}(x-\epsilon)_+.$$

Again, choosing an approximation $\phi_L$ of $\phi$ such that  $\phi_L(L)=L$, the proof uses the maximum principle and the fact that $v(x)$ is a subsolution.
\end{proof}

\section{Entropy dissipation inequality}\label{sec:entropy}

As it was seen in \cite{MR2065377,MR2162224,MR2536450,MR2832638} the
\emph{general relative entropy principle} applies to solutions of
\eqref{eq:growth-frag}. We remind that we use the entropy $H[g|G]$
defined in (\ref{entropydefinition}), with dissipation $D[g|G]$ given
by (\ref{entropydissipation}).  We recall that
\begin{equation*}
\frac{d}{dt}H[g|G]=-D[g|G] \le 0.
\end{equation*}
For the proof of the entropy inequality we will use \cite[Lemma
2.2]{MR2832638} with $\zeta(x)\equiv 1$. We need to check its hypotheses.
\begin{lem}
  Assume that
  Hypotheses~\ref{hyp:self-similar-frag}-\ref{hyp:asymptotics} and
  \ref{hyp:p-bounded} are satisfied with $\gamma \neq 0$. Given $M>1$
  there exists $K>0$ and $R>1$ such that the profiles $\phi$ and $G$
  satisfy the relations
  \begin{align}
    0\leq G(x)\leq K&
    \quad (x>0),
    \label{const1}
    \\
    \int_{Rx}^{\infty}G(y)\phi(y)\,dy \leq KG(x)&   \quad (x>M),
    \label{const2}
    \\
    \phi(y)\leq K\phi(z)& \quad (\max\lbrace2RM,Rz\rbrace<y<2Rz).
    \label{const3}
  \end{align}
\end{lem}
\begin{proof}
  The bound \eqref{const1} on $G$ is true because of
  Theorem~\ref{thm:G-at-infinity} and Theorem~\ref{thm:G-at-zero} with
  $\mu=1.$ For the second bound, we have due to l'H\^opital's rule and
  using Theorem~\ref{thm:G-at-infinity}
  \begin{align*}
    \int_{Rx}^{\infty}G(y)\phi(y)\,dy\leq\,&  K\int_{Rx}^{\infty}y^{1+\xi-\alpha}e^{-\Lambda(y)}\,dy\sim K x^{1+\xi-\gamma}e^{-\Lambda(Rx)}\\
    \leq\, & K x^{\xi-\alpha}e^{-\Lambda(x)} \sim KG(x).
  \end{align*}

  Finally, \eqref{const3} is a consequence of
  Theorem~\ref{thm:allbounds}.
\end{proof}
Moreover, for proving the entropy-entropy dissipation inequality, we
will need the following bounds, similar to those required in
\cite[Theorem 2.4]{MR2832638}.
\begin{lem}
  \label{lem:bounds-Kbxy}
  Suppose that the coefficients satisfy Hypotheses~{\rm
    \ref{hyp:self-similar-frag}-\ref{hyp:asymptotics}} and
  \ref{hyp:p-bounded} with one of the following additional conditions
  on the exponents $\gamma_0$ and $\alpha_0$:
  \begin{itemize}
 \item either $\alpha_0>1$,
 \item or $\alpha_0=1$ and $\gamma_0\leq1+\lambda/\tau_0,$
 \item or $\alpha_0<1$ and $\gamma_0\leq2-\alpha_0.$
  \end{itemize}
  Let $G$ and $\phi$ be the stationary profiles for the problems
  \eqref{eq:G} and \eqref{eq:phi}.  Then we can choose constants
  $K,M>0$ and $R>1$ such that the profiles $\phi$ and $G$ satisfy
  \begin{itemize}
  \item If $\gamma >0$ then
    \begin{align}
      G(x)\phi(y)\leq K b(y,x)& \qquad(0<x<y<\max\lbrace 2Rx, 2RM\rbrace),\label{condD21}\\
      y^{-1} \leq K b(y,x)& \qquad(y>M,\,\,y>x>0).\label{condD22}
    \end{align}
  \item If $\gamma <0$ then
    \begin{equation}
      G(x)\phi(y)\leq K b(y,x) \qquad(0<x<y)\label{condD21-D22}.
    \end{equation}
  \end{itemize}
\end{lem}

\begin{proof}
  {\bf Step 1: $\gamma > 0$, $y\leq2RM$ and $x<y.$} We need to estimate
  $G(x)\phi(y)$ at the limit $x<y\to0.$ Assume for the moment that
  $\alpha_0 \leq 1$. Using Theorem~\ref{thm:G-at-zero} ($G(x)\sim C
  x^{1-\alpha_0}$, notice that due to Hypothesis~\ref{hyp:p-bounded}
  one has $\mu=1$) and Theorem~\ref{thm:phi-at-zero} to bound $G(x)$
  and $\phi(y)$ respectively, we have
  \[ G(x)\phi(y)\leq C x^{1-\alpha_0}e^{\Lambda(y)}\leq
  C'\,y^{1-\alpha_0}\] since $\alpha_0\leq1.$ Then under the condition
  $\gamma_0\leq2-\alpha_0$ and from Hypothesis~\ref{hyp:p-bounded}, we
  get
  \begin{equation}
    \label{eq:Gphi-bound}
    G(x)\phi(y)\leq Cy^{\gamma_0-1}\leq Kb(y,x).
  \end{equation}
  When $\alpha_0=1,$
  we can do better since in this case we have necessarily $\gamma_0>0$
  and
  \[\Lambda(y)\underset{y\to0}{\sim}\frac{\lambda}{\tau_0}\ln(y).\]
  Thus we can write
  \[G(x)\phi(y)\leq Cy^{\lambda/\tau_0}\] and we obtain the bound
  $G(x)\phi(y)\leq Kb(y,x)$ from Hypothesis~\ref{hyp:p-bounded} as
  soon as $\gamma_0\leq1-\frac{\lambda}{\tau_0}.$

  When $\alpha_0 > 1$ we know that $\phi(y)$ decays exponentially as
  $y \to 0$, and one easily sees that the bound \eqref{eq:Gphi-bound}
  holds without any restriction on $\gamma_0$.
  
  \
  
  \noindent{\bf Step 2: $\gamma > 0$ and $2RM<y\leq2Rx.$}
  We need to estimate $G(x)\phi(y)$ at the limit $2Rx\geq
  y\to+\infty.$ Using \eqref{eq:bounds-gamma-positive} and
  \eqref{eq:G-asymptotic-x-large} we have
  \begin{align*}
   G(x)\phi(y)&\leq C(1+y)x^{\xi-\alpha}e^{-\Lambda(x)}\\
&\leq C'(1+y)y^{\xi-\alpha}e^{-\Lambda(y/2R)}\\
&\leq C''y^{\gamma-1}
\end{align*}
where $C''$ depends on $\alpha$, $\gamma$ and $R$. We conclude by
using Hypothesis~\ref{hyp:p-bounded}.

\

\noindent{\bf Step 3: $\gamma > 0$, $y>M$ and $y>x>0.$}
Since $\mu=1$ by assumption, we know from Theorems~\ref{thm:G-at-zero}
and \ref{thm:G-at-infinity} that $G(x)$ is bounded. When $\gamma>0,$
we observe first that $y^{-1}\leq Cy^{\gamma-1}$ and we conclude that
\eqref{condD22} holds true by using Hypothesis~\ref{hyp:p-bounded}.

\noindent{\bf Step 4: $\gamma < 0$.} When $\gamma<0$ we have from
Theorems~\ref{thm:allbounds}, \ref{thm:phi-at-zero} and
Hypothesis~\ref{hyp:p-bounded} that $\phi(y)\leq Cy^{\gamma-1}\leq
Kb(y,x)$ for all $0 < x < y$.
\end{proof}

At this point, we have all the tools to prove the entropy - entropy
dissipation inequality.
\begin{proof}[{\bf Proof of Theorem~\ref{thm:entropy_inequality}}]
From \cite[Lemma 2.1]{MR2832638} one can rewrite the entropy as follows
\begin{equation}
 D_2[g|G] :=
 \int_{0}^{\infty}\int_{x}^{\infty}
 \phi(x)G(x)\phi(y)G(y)(u(x)-u(y))^2 \dy \dx
 = H[g|G].
\end{equation}
If one looks at the integrand, one realizes that $D$ and $D_2$ have both $\phi(x)$ and $G(y)$ as a common terms. So we would like to compare and check that
\begin{equation}\label{eq:comparison}
 G(x)\phi(y) \le K b(y,x).
\end{equation}
We will denote by $C$ any constant depending on $G$, $\phi$, $K$, $M$, or $R$, but not on $g$.  We now distinguish two cases.\\

{\bf Case $\gamma <0$.}
The relation \eqref{eq:comparison} is satisfied due to \eqref{condD21-D22}. So we can compare pointwise the integrands of $D_2[g|G]$ with $D[g|G]$ and the inequality \eqref{entropy1} holds.\\

{\bf Case $\gamma>0$.} For proving the case $\gamma >0$ we follow the same argument as in \cite[Theorem 2.4]{MR2832638}. We start by rewriting $D_2[g|G]$ as follows:
\begin{equation*}
D_2[g|G] = D_{2,1}[g|G] + D_{2,2}[g|G],
\end{equation*}
where
\[
   D_{2,i}:=\iint_{A_i} \phi(x)G(x)\phi(y)G(y)(u(x)-u(y))^2 \dy \dx
\]
with $A_1:=\lbrace (x,y)\in \mathbb{R}_+^2 \,:\, y>x\,\,,\,\, y\leq RM\, \mbox{or}\,\, y < Rx\rbrace$ and $A_2=A_1^c$. For the first term and thanks to \eqref{condD21} one has
\begin{align}
D_{2,1}[g|G]\leq& \int_{0}^{\infty}\int_{x}^{\max\lbrace
  2Rx,2RM\rbrace}Kb(y,x)\phi(x)G(y)(u(x)-u(y))^2 \dy \dx
\nonumber\\
\leq & K D[g|G].\label{D21bound}
\end{align}
For the other term, what we have is
\begin{align}
D_{2,2}[g|G] \leq& C\int_{0}^{\infty}\int_{\max\lbrace x,
  M\rbrace}^{\infty}y^{-1}\phi(x)G(y)(u(x)-u(y))^2 \dy \dx\nonumber\\
\leq& C\,K \int_{0}^{\infty}\int_{\max\lbrace x,  M\rbrace}^{\infty}
b(y,x)\phi(x)G(y)(u(x)-u(y))^2 \dy\dx
\nonumber\\
\label{D22bound}
\leq& C\,K\,D[g|G],
\end{align}
where in the first inequality we applied \cite[Lemma 2.2]{MR2832638} with the bounds given in Lemma~\ref{lem:bounds-Kbxy} and for the second one we used \eqref{condD22}.
The proof concludes by gathering \eqref{D21bound} and \eqref{D22bound}.
\end{proof}


\appendix

\section{Approximation procedures}\label{sec:app:approximation}

To prove the estimates on the dual eigenfunction $\phi,$ we use a
truncated problem.  More precisely, we use alternatively one of the
following ones, which differ only in their boundary condition
\begin{equation}
  \label{eq:phi:truncated}
  \left\{\begin{array}{l}
      -\tau(x)\partial_x\phi_L(x)+(B(x)+\lambda_L)\phi_L(x)
      ={\mathcal L}^*_+(\phi_L)(x)
      \quad \text{ for } x \in (0,L),
      \vspace{2mm}\\
      \phi_L(L)=0\quad\text{or}\quad\phi_L(L)=\delta>0\quad\text{or}\quad\phi_L(L)=\delta L,\\
      \displaystyle\phi_L\geq0,\qquad\int_0^L G(x)\phi_L(x)\,dx=1.
    \end{array}\right.
\end{equation}
The following lemma ensures that these truncations converge to the accurate limit when $L\to+\infty.$
\begin{lem}\label{lem:convergence_truncated}
There exists $L_0>0$ such that for each $L\geq L_0$ the problem~\eqref{eq:phi:truncated} has a unique solution $(\lambda_L,\phi_L)$ with $\lambda_L>0$ and $\phi_L\in W^{1,\infty}_{loc}(\R_+).$
Moreover we have
$$\lambda_L\xrightarrow[L\to+\infty]{}\lambda,$$
$$\forall A>0,\quad \phi_L\xrightarrow[L\to+\infty]{}\phi\quad\text{uniformly on }[0,A).$$
\end{lem}

\begin{proof} We start with the case $\phi_L(L)=0$ by following the method in \cite{MR2652618}.
Define for $\eta>0$
\begin{equation}\label{eq:tau_eta}\tau_\eta(x):=\left\{\begin{array}{ll}
\eta&\text{for}\ 0<x<\eta,\\
\tau(x)&\text{for}\ x>\eta.
                       \end{array}\right.
\end{equation}
Then consider for $\epsilon>0$ and $L>0$ the truncated (and
regularized) eigenvalue problem on $[0,L]$
\begin{equation}
  \label{eq:truncated1}
  \left \{ \begin{array}{l}
\displaystyle \frac{\partial}{\partial x} (\tau_\eta(x) G_L(x)) + ( B(x) + \lambda_L) G_L(x) = \int_0^L b(y,x) G_L(y)\,dy,
\vspace{2mm}\\
\tau_\eta G_L(0)=\epsilon\int_0^L G_L(y)\,dy,\qquad G_L(x)>0 , \qquad \int_0^L G_L(x)dx =1,
\\
\\
\displaystyle -\tau_\eta(x) \frac{\partial}{\partial x} \phi_L(x) + ( B(x) + \lambda_L) \phi_L(x) - \int_0^L b(x,y) \phi_L(y)\,dy = \tau_\eta(0)\epsilon\phi_L(0),
\vspace{2mm}\\
\phi_L(L)=0, \qquad \phi_L(x)>0 , \qquad \int_0^L \phi_L(x)G_L(x)dx =1.
\end{array} \right.
\end{equation}

Notice that in this problem, the eigenelements $(\lambda_L,G_L,\phi_L)$ depend on $\eta,$ and $\epsilon$ and should be denoted
$(\lambda_L^{\eta,\epsilon},G_L^{\eta,\epsilon},\phi_L^{\eta,\epsilon}).$
We forget here the superscripts for the sake of clarity.

The existence of a solution to Problem~\eqref{eq:truncated1} is proved in the Appendix of~\cite{MR2652618} by using the Krein-Rutman theorem.
Then we need to pass to the limit $\eta,\,\epsilon\to0.$
The uniform estimates in~\cite{MR2652618} allow to do so, provided that $\lambda_L^{\eta,\epsilon}$ is positive for all $\eta,\,\epsilon.$
In~\cite{MR2652618} this condition is ensured for $L$ large enough under the constraint that $\epsilon L$ is a fixed constant, which means that $L=L(\epsilon)$ tends to $+\infty$ as $\epsilon\to0.$
Here we want to pass to the limit $\epsilon\to0$ for a fixed positive value of $L.$
For this we prove the existence of a constant $L_0>0$ such that $\lambda_L^{\eta,\epsilon}>0$ for all $\eta,\,\epsilon\geq0$ and all $L\geq L_0.$

\

Assume by contradiction that $\lambda_L\leq0.$ Then we have by integration of the direct eigenequation between $0$ and $x<L$
\begin{align*}
0&\geq \lambda\int_0^x G(y)\,dy\\
&=-\tau(x)G(x)-\int_0^x B(y)G(y)\,dy+\int_0^x\int_z^L b(y,z)G(y)\,dy\,dz\\
&=-\tau(x)G(x)+(\pi_0-1)\int_0^x B(y)G(y)\,dy+\int_x^L\left(\int_0^x b(y,z)\,dz\right)G(y)\,dz.
\end{align*}
We assume that $b(y,x)=\frac{B(y)}{y}p\bigl(\frac xy\bigr)$ with $\int_0^1 p(h)\,dh=\pi_0>1.$
Thus, for $p$ bounded, there exists $s\in(0,1)$ such that $\int_0^s p(h)\,dh\geq\pi_0-1.$
For $L\geq y\geq x\geq sL,$ we have
\begin{multline*}
\int_0^x b(y,z)\,dz=B(y)\int_0^{\frac xy}p(h)\,dh\\
\geq B(y)\int_0^{\frac{sL}y}p(h)\,dh\geq B(y)\int_0^{s}p(h)\,dh\geq(\pi_0-1)B(y),
\end{multline*}
so for all $x\geq sL$
$$0\geq-\tau(x)G(x)+(\pi_0-1)\int_0^L B(y)G(y)\,dy$$
which leads to
$$B(x)G(x)\geq(\pi_0-1)\frac{B(x)}{\tau(x)}\int_0^L B(y)G(y)\,dy\geq(\pi_0-1)\frac{B(x)}{\tau(x)}\int_{sL}^L B(y)G(y)\,dy$$
and finally, by integration on $[sL,L],$
\begin{equation}\label{eq:contradiction}(\pi_0-1)\int_{sL}^L\frac{B(y)}{\tau(y)}\,dy\leq1.\end{equation}
We have from Hypothesis~\ref{hyp:asymptotics} that
$$\exists\,A>0,\qquad\forall x\geq A,\quad \frac{xB(x)}{\tau(x)}>\frac{1}{(\pi_0-1)|\ln(s)|}$$
so, for $L\geq\frac As,$ we obtain
$$(\pi_0-1)\int_{sL}^L\frac{B(y)}{\tau(y)}\,dy > \frac{1}{|\ln(s)|}\int_{sL}^L\frac{1}{y}\,dy=1$$
which contradicts~\eqref{eq:contradiction}.
Finally, $\lambda_L>0$ for all $L\geq L_0:=\frac As.$

\

We have proved the existence of solution for Problem~\eqref{eq:phi:truncated} in the case $\phi_L(L)=0$ and we know that
$$\lambda_L\xrightarrow[L\to+\infty]{}\lambda,$$
$$G_L\xrightarrow[L\to+\infty]{}G\qquad\text{in}\ L^1(\R_+),$$
$$\forall A>0,\quad \phi_L\xrightarrow[L\to+\infty]{}\phi\qquad\text{uniformly on }[0,A).$$
Now we use this result to treat the cases $\phi_L(L)=\delta>0$ and $\phi_L(L)=\delta L.$
Since $\delta>0,$ we can prove by using the Krein-Rutman theorem the existence of a solution to Problem~\eqref{eq:phi:truncated}.
To prove the convergence of $\lambda_L$ to $\lambda,$ we integrate the equation on $\phi_L$ multiplied by $G$ and we obtain
$$\lambda-\lambda_L=\tau(L)G(L)\phi_L(L).$$
We know from estimates on $G$ that $\tau(L)G(L)L\to0$ when $L\to+\infty,$ which ensures the convergence of $\lambda.$
Because $\lambda>0,$ it also ensures the existence of $L_0$ such that $\lambda_L>0$ for $L\geq L_0,$ which allows to prove the convergence of $\phi_L$ to $\phi$ locally uniformly (see~\cite{MR2652618} for details).
\end{proof}

\section{Laplace's method}
\label{sec:Laplace}

Laplace's method (see \cite[II.1, Theorem 1]{MR1851050} for
example) gives a way to calculate the asymptotic behavior of integrals
which contain an exponential term with a large factor in the
exponent. We give here a result of this kind, with conditions which
are adapted to the situation encountered in Section \ref{sec:G}.

\begin{lem}
  \label{lem:Laplace3}

  Take $x_0, D_0 \in \R$. Assume that $g: [x_0,+\infty) \to \R$ is a
  measure and $h:[x_0,+\infty) \times [D_0,+\infty) \to \R$ a
  measurable function satisfying
  \begin{gather}
    \label{eq:g-at-0--3}
    g(x) \sim g_0 (x-x_0)^\sigma
    \quad
    \text{ as $x \to x_0$, for some $g_0 \neq 0$ and $\sigma > -1$,}
    \\
    \label{eq:h-at-0--3}
    \begin{split}
      h(x,D) &- h(x_0,D) \sim h_0 (x - x_0)^\omega
      \\ & \quad \text{ as $x \to
        x_0$ and $D \to +\infty$, for some $h_0, \omega > 0$,}
    \end{split}
    \\
    \label{eq:integral-finite--3}
    \int_{x_0}^\infty |g(x)| e^{-D_0 h(x,D)} \,dx < C_0
    \quad \text{ for some $C_0 \geq 0$ and all $D \geq D_0$.}
  \end{gather}
  Assume also that for all $D \geq D_0$, the function $x \mapsto
  h(x,D)$ (with $D$ fixed) attains its unique global minimum at
  $x=x_0$, in the following strong sense: there exists a nondecreasing
  strictly positive function $\theta:(0,+\infty) \to (0,+\infty)$ such
  that
  \begin{equation}
    \label{eq:strong-maximum--3}
    h(x,D) - h(x_0,D) \geq \theta(x-x_0) \quad \text{ for all $x >
      x_0$ and all $D \geq D_0$}.
  \end{equation}
  Then, as $D \to +\infty$,
  \begin{equation}
    \label{eq:I-asymptotic--3}
    \int_{x_0}^\infty e^{-Dh(x,D)} g(x) \,dx
    \sim
    g_0 D^{\frac{-1-\sigma}{\omega}}
    e^{-D h(x_0,D)} \int_0^\infty x^\sigma e^{-h_0 x^\omega} \,dx.
  \end{equation}
  The constants implicit in \eqref{eq:I-asymptotic--3} depend only on
  the constants implicit or explicit in
  \eqref{eq:g-at-0--3}--\eqref{eq:strong-maximum--3}.
\end{lem}

Some remarks on the conventions used above are in order. Although $g$
is a measure we denote it as a function in the expressions in which it
appears. For example, integrals in which $g$ appears should be
considered as integrals with respect to the measure $g$. Also, in
equation (\ref{eq:g-at-0--3}), it is understood that close to $x_0$
the measure $g$ is equal to a function, and the asymptotic
approximation (\ref{eq:g-at-0--3}) holds.

\begin{proof}
  First of all, by translating $g$ and $h$ we may consider always that
  $x_0 = 0$. We may also assume that $h(x_0,D) = 0$ for all $D \geq D_0$,
  as otherwise one obviously obtains the additional factor
  $e^{-Dh(x_0,D)}$.

  An important part of the argument is based on the observation that
  if one excludes a small region close to $0$, then the rest of the
  integral decreases fast as $D \to +\infty$: from \eqref{eq:h-at-0--3}
  and \eqref{eq:strong-maximum--3} we deduce that for some $\rho > 0$
  \begin{equation}
    \label{eq:h-bound--3}
    h(x,D) \geq \rho \min\{1, x^\omega\}
    \quad \text{ for all $x \geq 0$, $D \geq D_0$}.
  \end{equation}
  Then, for $D \geq D_0$ and $0 < \epsilon < 1$ we have from
  (\ref{eq:h-bound--3}):
  \begin{multline*}
    \left|
      \int_\epsilon^\infty g(x)  e^{-D h(x,D)} \,dx
    \right|
    \\
    \leq
    e^{-(D-D_0) \rho \epsilon^\omega}
    \int_\epsilon^\infty |g(x)| e^{-D_0 h(x,D)} \,dx
    \leq
    C_0 e^{-(D-D_0) \rho \epsilon^\omega},
  \end{multline*}
  due to
  (\ref{eq:integral-finite--3}). If we take $\epsilon :=
  D^{-\frac{1}{2\omega}}$ then for all $D > D_0$ we have
  \begin{equation*}
    \left|
      \int_{D^{-\frac{1}{2\omega}}}^\infty g(x) e^{-D h(x,D)} \,dx
    \right|
    \leq
    C_0 e^{-\left(\sqrt{D} - \frac{D_0}{\sqrt{D}} \right) \rho}
    .
  \end{equation*}
  This quantity decreases faster than any power of $D$ as $D \to
  +\infty$.

  For the remaining part of the integral, since we are integrating in
  a region which is closer and closer to $0$ it is easy to see due to
  (\ref{eq:g-at-0--3}) and (\ref{eq:h-at-0--3}) that for all $\epsilon
  > 0$ there exists $D_\epsilon > 0$ such that
  \begin{multline}
    \label{eq:s1}
    \int_0^{D^{-\frac{1}{2\omega}}} (1-\epsilon) g_0 x^\sigma e^{-D
      (1+\epsilon) h_0 x^\omega} \,dx
    \leq
    \int_0^{D^{-\frac{1}{2\omega}}} g(x) e^{-D h(x,D)} \,dx
    \\
    \leq
    \int_0^{D^{-\frac{1}{2\omega}}} (1+\epsilon) g_0 x^\sigma e^{-D
      (1-\epsilon) h_0 x^\omega} \,dx
  \end{multline}
  for all $D > D_\epsilon$. Through the change of variables $z = x
  D^{1/\omega}$ we see that
  \begin{multline*}
    \int_0^{D^{-\frac{1}{2\omega}}} (1-\epsilon) g_0 x^\alpha e^{-D
      (1+\epsilon) h_0 x^\omega} \,dx
    \\=
    (1-\epsilon) g_0 D^{\frac{-1-\sigma}{\omega}}
    \int_0^{D^{\frac{1}{2\omega}}} z^\sigma e^{- (1+\epsilon)h_0 z^\omega} \,dz
    \\
    \sim
    (1-\epsilon) g_0 D^{\frac{-1-\sigma}{\omega}}
    \int_0^{\infty} z^\sigma e^{- (1+\epsilon)h_0 z^\omega} \,dz,
  \end{multline*}
  where the `$\sim$' sign denotes asymptotics as $D \to
  +\infty$. Carrying out a similar calculation for the last integral
  in \eqref{eq:s1} and letting $\epsilon \to 0$ we deduce our result.
\end{proof}

For the next result we recall that $\gamma_+=\max\{0,\gamma\}$ and
$\zeta$ is defined by~\eqref{eq:zeta}.

\begin{lem}
  \label{lem:asymptotics-2nd}
  Assume Hypotheses
  \ref{hyp:self-similar-frag}--\ref{hyp:asymptotics-2nd}. There is
  $\epsilon > 0$ such that
  \begin{equation*}
    \int_{x_0}^x e^{\Lambda(y)} y^k \,dy
    =
    \zeta^{-1}
    x^{k-\gamma_++\alpha} e^{\Lambda(x)}
    + \underset{x\to+\infty}{O}(x^{k-\gamma_++\alpha-\epsilon}) e^{\Lambda(x)}.
  \end{equation*}
\end{lem}

\begin{proof}
  We use l'H\^opital's rule to calculate the limit as $x \to +\infty$ of
  \begin{equation*}
    \frac{\int_{x_0}^x e^{\Lambda(y)} y^k \,dy
      -
      \zeta^{-1} x^{k-\gamma_++\alpha} e^{\Lambda(x)}
    }{x^m e^{\Lambda(x)}}.
  \end{equation*}
  Differentiating both the numerator and denominator we find that this
  limit is the same as the limit as $x \to +\infty$ of
  \begin{equation}
    \label{eq:fraction}
    \frac{x^k
      -
      \zeta^{-1}(k-\gamma_++\alpha) x^{k-\gamma_++\alpha-1}
      - \zeta^{-1} x^{k-\gamma_++\alpha} \frac{B(x)+\lambda}{\tau(x)}
    }{m x^{m-1} + x^{m} \frac{B(x)+\lambda}{\tau(x)}}
    =: \frac{T_N(x)}{T_D(x)}.
  \end{equation}
  Using \eqref{eq:B-asymptotics-infty-2nd} and
  \eqref{eq:tau-asymptotics-infty-2nd}, we obtain that
  \begin{equation*}
     x^k
    - \zeta^{-1} x^{k-\gamma_++\alpha} \frac{B(x)+\lambda}{\tau(x)}
    = O(x^{k-\delta})
  \end{equation*}
  for some $\delta > 0$. Observing that $\gamma_+-\alpha+1 > 0$
  and calling $\epsilon := \min\{\delta, \gamma_+-\alpha+1 \} > 0$
  we have
  \begin{equation*}
    T_N(x) = O(x^{k-\epsilon}).
  \end{equation*}
  In a similar way,
  \begin{equation*}
    T_D(x) = x^{m+\gamma_+-\alpha} + O(x^{m+\gamma_+-\alpha-\delta}),
  \end{equation*}
  so from \eqref{eq:fraction} we obtain that the limit is 0 whenever
  \begin{equation*}
    m + \gamma_+ - \alpha > k-\epsilon,
    \quad \text{ i.e., }
    m > k-\gamma_++\alpha-\epsilon.
  \end{equation*}
  This shows the result.
\end{proof}

We now use this to prove an estimate which is needed in Section
\ref{sec:G}:

\begin{lem}
  \label{lem:Laplace}
  Assume Hypotheses
  \ref{hyp:self-similar-frag}--\ref{hyp:p-at-borders} with $p_1 > 0$,
  and take $k \in \R$. Then,
  \begin{equation}
    \label{eq:Laplace}
    \int_{x_0}^x e^{\Lambda(y)} y^k b(x,y) \,dy
    \underset{x\to+\infty}{\sim}
    p_1 B_\infty \zeta^{-1-\nu}\Gamma(\nu+1) x^{k+\gamma -(\gamma_+-\alpha+1)(1+\nu)}e^{\Lambda(x)}.
  \end{equation}
  If we also assume Hypothesis \ref{hyp:asymptotics-2nd} and $\nu = 0$
  (and now we allow any $p_1 \geq 0$) then there is $\epsilon > 0$
  such that
  \begin{multline}
    \label{eq:Laplace-2nd}
    \int_{x_0}^x e^{\Lambda(y)} y^k b(x,y) \,dy
    \\
    =
    p_1B_\infty\zeta^{-1}\, x^{k +\gamma - \gamma_+ + \alpha - 1}e^{\Lambda(x)}
    + \underset{x\to+\infty}O(x^{k +\gamma - \gamma_+ + \alpha - 1 -\epsilon}) e^{\Lambda(x)}.
  \end{multline}
\end{lem}

\begin{proof}
  We call $p_*$ and $p^*$, respectively, the parts of the measure $p$
  on the intervals $[0,1/2)$ and $[1/2,1]$, i.e., $p_* := p
  \,\1_{[0,1/2)}$ and $p^* := p \,\1_{[1/2,1]}$. With this we
  break the integral we want to estimate in two parts:
  \begin{multline*}
    I(x) := \int_{x_0}^x e^{\Lambda(y)} y^k
    p\left(\frac{y}{x}\right) \,dy
    \\
    = \int_{x_0}^x e^{\Lambda(y)} y^k p_*\left(\frac{y}{x}\right) \,dy
    +  \int_{x_0}^x e^{\Lambda(y)} y^k p^*\left(\frac{y}{x}\right) \,dy
    =: I_*(x) + I^*(x).
  \end{multline*}
  The first part, $I_*$, can be estimated by
  \begin{multline*}
    |I_*(x)| = \int_{x_0}^x e^{\Lambda(y)} y^k
    p_*\left(\frac{y}{x}\right)\,dy
    \leq e^{\Lambda(x/2)}
    \int_{x_0}^x y^k p_*\left(\frac{y}{x}\right)\,dy
    \\
    \leq
    e^{\Lambda(x/2)}
    \max\{x^k, x_0^k\}
    \int_{x_0}^x p_*\left(\frac{y}{x}\right)\,dy
    \leq
    \pi_0\, x \,e^{\Lambda(x/2)}
    \max\{x^k, x_0^k\}.
  \end{multline*}
  Since we will show that $I^*(x)$ behaves as given in the statement,
  this term is of lower order (since $\Lambda(x)$ is asymptotic
  to a positive power of $x$ as $x \to +\infty$) and can be
  disregarded.

  For $I^*$ we make the change of variables $z = y/x$ and denote $D :=
  x^{\gamma_+-\alpha+1}$ to obtain
  \begin{multline}
    \label{eq:asymp-pf-1}
    I^*(x) := \int_{x_0}^x e^{\Lambda(y)} y^k p^*\left(\frac{y}{x}\right) \,dy
    =
    x^{k+1} \int_{\max\left\{ \frac{x_0}{x}, \frac{1}{2} \right\}}^1
    e^{\Lambda(xz)} z^k p\left(z\right) \,dz
    \\
    =
    x^{k+1}
    \int_{\max\left\{ \frac{x_0}{x}, \frac{1}{2} \right\}}^1
    \exp\left(-D h(z,D) \right)
    g(z) \,dz
  \end{multline}
  with
  \begin{equation*}
    h(z,D) := -\frac{1}{D} \Lambda\left( x z\right)
    = -\frac{1}{D} \Lambda\left( D^{\frac{1}{\gamma_+-\alpha+1}} z\right),
    \qquad
    g(z) := z^k p\left(z\right).
  \end{equation*}
  Now, the asymptotics in $D$ of the integral in \eqref{eq:asymp-pf-1}
  can be obtained from Lemma \ref{lem:Laplace3} with $x_0 = 1$. Let us
  see that $h$ and $g$ indeed satisfy the needed hypotheses. The
  property \eqref{eq:g-at-0--3} is satisfied with $g_0=p_1$ and $\sigma=\nu$ due to Hypothesis
  \ref{hyp:p-at-borders}, and to show \eqref{eq:h-at-0--3} we write
  (with asymptotics notation understood to be for $z \to 1$ and $D \to
  +\infty$)
  \begin{multline}
    \label{eq:asymp-pf-2}
    h(z,D) - h(1,D)
    =
    \frac{1}{D} \left(\Lambda(x) - \Lambda\left( x z\right) \right)
    =
    \frac{1}{D} \int_{xz}^x \frac{\lambda + B(u)}{\tau(u)} \,du
    \\
    \sim
    \frac{\zeta}{\gamma_+ - \alpha + 1} \frac{1}{D}
    x^{\gamma_+ - \alpha + 1} (1 - z^{\gamma_+ - \alpha + 1})
    =
    \frac{\zeta}{\gamma_+ - \alpha + 1} (1 - z^{\gamma_+ - \alpha + 1})
    \sim
    \zeta(1-z),
  \end{multline}
  which corresponds to $h_0=\zeta$, $\omega = 1$ in Lemma
  \ref{lem:Laplace3}. For \eqref{eq:integral-finite--3} we write
  \begin{multline*}
    \int_{\max\left\{ \frac{x_0}{x}, \frac{1}{2} \right\}}^1
    \exp\left(-D_0 h(z,D) \right)
    g(z) \,dz
    \\
    \leq
    \int_{\frac{1}{2}}^1
    \exp\left(\frac{D_0}{D} \Lambda(x z) \right)
    z^k p(z) \,dz
    \\
    \leq
    \exp\left(\frac{D_0}{D} \Lambda(x) \right) \int_{\frac{1}{2}}^1
    z^k p(z) \,dz
    \leq C_0
  \end{multline*}
  for some $C_0 > 0$ (which in particular depends on $k$), since $x
  \mapsto \Lambda(x) / D = \Lambda(x) / x^{\gamma_+-\alpha+1}$ is
  bounded for $x > 1$. This gives
  \eqref{eq:integral-finite--3}. Obviously $z \mapsto h(z,D)$ attains
  its minimum at $z = 1$, and \eqref{eq:strong-maximum--3} is a
  consequence of \eqref{eq:asymp-pf-2} and the fact that
  $h(z,D)-h(1,D)$ is decreasing in $z$ for all $D$.

  We may then apply Lemma \ref{lem:Laplace3} to obtain
  \begin{multline*}
    I^*(x)
    \sim
    p_1
    x^{k+1}
    D^{-1-\nu} e^{\Lambda(x)}
    \int_0^\infty z^{\nu} e^{-\zeta z} \,dz
    \\
    =p_1\zeta^{-1-\nu}\Gamma(1+\nu)
    x^{k+1 - (\gamma_+-\alpha+1)(1+\nu)}
    e^{\Lambda(x)}.
  \end{multline*}
  Since $I_*(x)$ was shown to be of lower order, this is enough to
  show \eqref{eq:Laplace}.

  \medskip
  Finally, in order to show \eqref{eq:Laplace-2nd}, we have
  \begin{equation*}
    \int_{x_0}^x e^{\Lambda(y)} y^k p\left(\frac{y}{x}\right) \,dy
    =
    \int_{x_0}^x e^{\Lambda(y)} y^k \left(
      p\left(\frac{y}{x}\right) - p_1 \right) \,dy
    +
    p_1 \int_{x_0}^x e^{\Lambda(y)} y^k \,dy.
  \end{equation*}
  For the first term, using \eqref{eq:p-asymptotics-2nd} and
  \eqref{eq:Laplace} we have
  \begin{equation*}
    \int_{x_0}^x e^{\Lambda(y)} y^k \left(
      p\left(\frac{y}{x}\right) - p_1 \right) \,dy
    =
    O(x^{k+1 -(\gamma_+-\alpha+1)(1+\delta)})
    e^{\Lambda(x)},
  \end{equation*}
  and for the second term, Lemma \ref{lem:asymptotics-2nd} gives
  \begin{equation*}
    \int_{x_0}^x e^{\Lambda(y)} y^k \,dy
    =
    \zeta^{-1}
    x^{k-\gamma_++\alpha} e^{\Lambda(x)}
    + O(x^{k-\gamma_++\alpha-\delta}) e^{\Lambda(x)}.
  \end{equation*}
  Since $\gamma_+-\alpha+1 > 0$, this shows the result.
\end{proof}


\section*{Acknowledgments.}
D. Balagu\'e and J.~A.~Ca\~nizo were supported by the projects
MTM2011-27739-C04-02 DGI-MICINN (Spain) and 2009-SGR-345 from
AGAUR-Generalitat de Catalunya.  The research of P. Gabriel was
supported by the French National Agency for Research through the
grants TOPPAZ ANR-09-BLAN-0218 and PAGDEG ANR-09-PIRI-0030.


\begin{thebibliography}{99}

\bibitem{MR2832638}
\newblock M.~J. C\'{a}ceres, J.~A. Ca\~{n}izo and S.~Mischler,
\newblock \emph{Rate of convergence to an asymptotic profile for the self-similar fragmentation and growth-fragmentation equations,}
\newblock J. Math. Pures Appl. (9), \textbf{96} (2011), 334--362.

\bibitem{MR2652618}
\newblock M.~Doumic Jauffret and P.~Gabriel,
\newblock \emph{Eigenelements of a general aggregation-fragmentation model,}
\newblock Math. Models Methods Appl. Sci., \textbf{20} (2010), 757--783.

\bibitem{MR2114413}
\newblock M.~Escobedo, S.~Mischler and M.~Rodr\'{\i}guez~Ricard,
\newblock \emph{On self-similarity and stationary problem for fragmentation and coagulation models,}
\newblock Ann. Inst. H. Poincar\'{e} Anal. Non Lin\'{e}aire, \textbf{22} (2005), 99--125.

\bibitem{GabrielPhD}
\newblock P.~Gabriel,
\newblock ``\'{E}quations de {T}ransport-{F}ragmentation et {A}pplications aux {M}aladies \`a {P}rions [Transport-{F}ragmentation {E}quations and {A}pplications to {P}rion {D}iseases],"
\newblock Ph.D thesis, Paris, 2011.

\bibitem{MR2536450}
\newblock P.~Lauren\c{c}ot and B.~Perthame,
\newblock \emph{Exponential decay for the growth-fragmentation/cell-division equation,}
\newblock Comm. Math. Sci., \textbf{7} (2009), 503--510.

\bibitem{MR0860959}
\newblock J.~A.~J. Metz and O.~Diekmann, eds.,
\newblock ``The Dynamics of Physiologically Structured Populations,"
\newblock Lecture notes in Biomathematics, \textbf{68}, Springer-Verlag, Berlin, 1986.

\bibitem{MR2250122}
\newblock P.~Michel,
\newblock \emph{Existence of a solution to the cell division eigenproblem,}
\newblock Math. Models Methods Appl. Sci., \textbf{16} (2006), 1125--1153.

\bibitem{MR2065377}
\newblock P.~Michel, S.~Mischler and B.~Perthame,
\newblock \emph{General entropy equations for structured population models and scattering,}
\newblock C. R. Math. Acad. Sci. Paris, \textbf{338} (2004), 697--702.

\bibitem{MR2162224}
\newblock P.~Michel, S.~Mischler and B.~Perthame,
\newblock \emph{General relative entropy inequality: An illustration on growth models,}
\newblock J. Math. Pures Appl. (9), \textbf{84} (2005), 1235--1260.

\bibitem{MR2270822}
\newblock B.~Perthame,
\newblock ``Transport Equations in Biology,"
\newblock Frontiers in Mathematics, Birkh\"auser Verlag, Basel, 2007.

\bibitem{MR2114128}
\newblock B.~Perthame and L.~Ryzhik,
\newblock \emph{Exponential decay for the fragmentation or cell-division equation,}
\newblock J. Differential Equations, \textbf{210} (2005), 155--177.

\bibitem{PS}
\newblock B.~Perthame and D.~Salort,
\newblock \emph{Distributed elapsed time model for neuron networks,}
\newblock in preparation.

\bibitem{MR1851050}
\newblock R.~Wong,
\newblock {``Asymptotic Approximation of Integrals,"}
\newblock Corrected reprint of the 1989 original, Classics in Applied Mathematics, \textbf{34}, Society for Industrial and Applied Mathematics, Philadelphia, PA, 2001.



\end{thebibliography}

\end{document}